\documentclass[10pt]{amsart}
\usepackage{graphicx}
\usepackage{amsmath,amssymb}
\usepackage{epstopdf}
\usepackage[]{geometry}
\usepackage{longtable}
\usepackage{booktabs}
\usepackage[foot]{amsaddr}
\usepackage{bm}
\usepackage{xcolor}
\usepackage{stackrel}
\usepackage{accents}
\usepackage{enumitem}
\usepackage{adjustbox}
\usepackage{subfig}
\usepackage{stmaryrd}
\usepackage{marvosym}
\usepackage{mathtools}
\usepackage{ulem}
\usepackage[pdftex,pdfpagelabels=true,bookmarksdepth=3]{hyperref}
\hypersetup{colorlinks=false,
            bookmarksnumbered=true,
            pdftitle = {Entropy stable numerical approximations for the isothermal and polytropic Euler equations},
pdfsubject = {},
pdfauthor = {A.Winters,C.Czernik,M.Schily,G.Gassner},
pdfkeywords = {}
 }
\newcommand{\avg}[1]{\left\{\hspace*{-3pt}\left\{#1\right\}\hspace*{-3pt}\right\}}
\newcommand{\gamavg}{\left\{\hspace*{-3pt}\left\{\rho\right\}\hspace*{-3pt}\right\}_{\!\gamma}}
\newcommand{\aAvg}{\overline{a^2}}
\newcommand{\jump}[1]{\ensuremath{\left\llbracket #1 \right\rrbracket}}
\newcommand{\pderivative}[2]{\frac{\partial #1}{\partial #2}}
\let\rho\varrho
\newcommand\spacevec[1]{\accentset{\,\rightarrow}{#1}}				
\newcommand\statevec[1]{\mathbf #1}							
\newcommand\contrastatevec[1]{\tilde{\mathbf #1}} 					
\newcommand\acclrvec[1]{\accentset{\,\leftrightarrow}{#1}}			
\newcommand\bigstatevec[1]{\acclrvec{{\mathbf #1}}}				
\newcommand\bigcontravec[1]{\acclrvec{\tilde{\mathbf #1}}} 			
\newcommand\threeMatrix[1]{\underline{ #1}}						
\newcommand\nineMatrix[1]{\mathsf{ #1}}							
\newcommand\matrixvec[1]{\mathcal #1}							
\newcommand{\ec}{{\mathrm{EC}}}								
\newcommand{\es}{{\mathrm{ES}}}								
\definecolor{chcol}{rgb}{0.4,0.,0.9}

\theoremstyle{plain}
\newtheorem{theorem}{Theorem}
\newtheorem{lemma}{Lemma}
\newtheorem{corollary}{Corollary}
\theoremstyle{remark}
\newtheorem{definition}{Definition}
\newtheorem{remark}{Remark}
\numberwithin{equation}{section}
\begin{document}

\title[ES numerical methods for isothermal and polytropic Euler]{Entropy stable numerical approximations for the isothermal and polytropic Euler equations}
%
%
\author[Winters]{Andrew~R.~Winters$^{1,*}$}
\address{$^1$Department of Mathematics; Computational Mathematics, Link\"{o}ping University, SE-581 83 Link\"{o}ping, Sweden}
\author[Czernik]{Christof Czernik}
\author[Schily]{Moritz B. Schily}
\address{$^2$Mathematical Institute, University of Cologne, Weyertal 86-90, 50931 Cologne, Germany}
\author[Gassner]{Gregor~J.~Gassner$^{3}$}
\address{$^3$Department for Mathematics and Computer Science; Center for Data and Simulation Science, University of Cologne, Weyertal 86-90, 50931, Cologne, Germany}
\email{andrew.ross.winters@liu.se}

%

\maketitle

\begin{abstract}

	In this work we analyze the entropic properties of the Euler equations when the system is closed with the assumption of a polytropic gas. In this case, the pressure solely depends upon the density of the fluid and the energy equation is not necessary anymore as the mass conservation and momentum conservation then form a closed system. Further, the total energy acts as a convex mathematical entropy function for the polytropic Euler equations. The polytropic equation of state gives the pressure as a scaled power law of the density in terms of the adiabatic index $\gamma$. As such, there are important limiting cases contained within the polytropic model like the isothermal Euler equations ($\gamma=1$) and the shallow water equations ($\gamma=2$). We first mimic the continuous entropy analysis on the discrete level in a finite volume context to get special numerical flux functions. Next, these numerical fluxes are incorporated into a particular discontinuous Galerkin (DG) spectral element framework where derivatives are approximated with summation-by-parts operators. This guarantees a high-order accurate DG numerical approximation to the polytropic Euler equations that is also consistent to its auxiliary total energy behavior. Numerical examples are provided to verify the theoretical derivations, i.e., the entropic properties of the high order DG scheme. 

\end{abstract}
\vspace{0.5cm}
\noindent\textbf{Keywords:} Isothermal Euler, Polytropic Euler, Entropy stability, Finite volume, Summation-by-parts, Nodal discontinuous Galerkin spectral element method
\section{Introduction}

	The compressible Euler equations of gas dynamics
	\begin{equation}\label{eq:compEuler}
	\begin{aligned}
		\rho_t + \spacevec{\nabla}\cdot(\rho\spacevec{v}\,) &= 0,\\[0.05cm]
		(\rho\spacevec{v}\,)_t + \spacevec{\nabla}\cdot\left(\rho\spacevec{v}\otimes\spacevec{v}\,\right) + \spacevec{\nabla} p &= \spacevec{0},\\[0.05cm]
		E_t + \spacevec{\nabla}\cdot\left(\spacevec{v}\,[E + p]\right) &= 0,
	\end{aligned}
	\end{equation}
	are a system of partial differential equations (PDEs) where the conserved quantities are the mass $\rho$, the momenta $\rho\spacevec{v}$, and the total energy $E = \frac{\rho}{2}\|\spacevec{v}\|^2 + \rho e$.  This is an archetypical system of non-linear hyperbolic conservation laws that have far reaching applications in engineering and natural sciences, e.g. \cite{kundu2008,leveque2002,whitham1974}. In three spatial dimensions, this system has five equations but six unknowns: the density $\rho\in\mathbb{R}^+$, the velocity components $\spacevec{v}=(v_1,v_2,v_3)\in\mathbb{R}^3$, the internal energy $e\in\mathbb{R}$, and the pressure $p\in\mathbb{R}^+$. Thus, in order to close the system, an \textit{equation of state} is necessary to relate thermodynamic state variables like pressure, density, and internal energy. Depending on the fluid and physical processes we wish to model the equation of state changes. Some examples include an ideal gas where $p\equiv p(\rho,e)$ or polytropic processes where $p\equiv p(\rho)$ \cite{kundu2008}.
	
	The connection between the equation of state, the fluid, and other thermodynamic properties is of particular relevance when examining the physical realizability of flow configurations. In particular, the \textit{entropy} plays a crucial role to separate possible flow states from the impossible \cite{cengel2014}. There is a long history investigating the thermodynamic properties of the compressible Euler equations through the use of mathematical entropy analysis for adiabatic processes \cite{harten1983,mock1980,Tadmor2003} as well as polytropic processes \cite{chen2005,serre1999}. In this analysis the mathematical entropy is modeled by a strongly convex function $s(\rho,e)$. There exist associated entropy fluxes, $\spacevec{f}^s$, such that the entropy function satisfies an additional conservation law 
	\begin{equation*}
		s_t + \spacevec{\nabla}\cdot\spacevec{f}^s = 0,
	\end{equation*}
	for smooth solutions that becomes an inequality 
	\begin{equation*}
		s_t + \spacevec{\nabla}\cdot\spacevec{f}^s \leq 0,
	\end{equation*}
in the presence of discontinuous solutions, e.g. shocks.
Note, we have adopted the convention common in mathematics that entropy is a decreasing quantity, e.g. \cite{Tadmor2003}.
	
	For numerical methods, discretely mimicking this thermodynamic behavior leads to schemes that are entropy conservative (or entropy stable) depending on the solutions smoothness \cite{harten1983,Tadmor2003}. Additionally, numerical approximations, especially schemes with higher order accuracy and low inbuilt numerical dissipation, that are thermodynamically consistent have a marked increase in their robustness \cite{chan2018,Gassner:2016ye,kuya2018,Winters2018}. Thus, the design and application of entropy stable approximations, particularly for the compressible Euler equations, have been the subject of ongoing research for the past 50 years, e.g. \cite{carpenter_esdg,chan2018,Chandrashekar2012,Chen2017,crean2018,fisher2013,fisher2013_2,Gassner:2016ye,harten1983,IsmailRoe2009,Ray2016,Tadmor2003}. A major breakthrough came with the seminal work of Tadmor \cite{tadmor1984} wherein he developed a general condition for a finite volume numerical flux function to remain entropy conservative. It was then possible to selectively add dissipation to the baseline numerical approximation and guarantee entropy stability.
	
	Many authors expanded on the entropy stability work of Tadmor, developing higher order spatial approximations through the use of WENO reconstructions \cite{Fjordholm2012_2,Fjordholm2016,Lefloch2002}, summation-by-parts (SBP) finite difference approximations \cite{crean2018,fisher2013,fisher2013_2}, or the discontinuous Galerkin spectral element method (DGSEM) also with the SBP property \cite{carpenter_esdg,chan2018,Chen2017,Gassner:2016ye,Gassner2017}. The latter two numerical schemes both utilize the SBP property that discretely mimics integration-by-parts. This allows a direct translation of the continuous analysis and entropy stability proofs onto the discrete level, see \cite{gassner_skew_burgers,svard2014} for details. However, the design of these entropy stable approximations (low-order or high-order) has focused on adiabatic processes for the compressible Euler equations.
	
	So, the main focus in this work is to design entropy conservative and entropy stable numerical methods for the polytropic Euler equations. As such, the mathematical entropy analysis is reinvestigated on the continuous level due to the selection of a different equation of state. This analysis also provides a roadmap to discrete entropy stability. We will show that isothermal limit ($\gamma = 1$) requires special considerations. The first contribution comes with the derivation of entropy conservative/stable numerical flux functions from Tadmor's finite volume condition. This includes a computationally affordable definition of the baseline entropy conservative numerical flux as well as an explicit definition of the average states where the dissipation terms should be evaluated. In particular, a special mean operator, which is a generalization of the logarithmic mean \cite{stolarsky1975}, is introduced. The second contribution takes the finite volume derivations and builds them into a high-order DGSEM framework that remains consistent to the laws of thermodynamics. Complete details on the entropy aware DGSEM are given by Gassner et al. \cite{Gassner:2016ye}.
	
	The paper is organized as follows: Sect.~\ref{sec:polyEuler} presents the polytropic Euler system and performs the continuous mathematical entropy analysis. The derivations are kept general as the isothermal Euler equations are a special case of the polytropic system. The finite volume discretization and entropy stable numerical flux derivations are given in Sect.~\ref{sec:discEntropy}. In Sect.~\ref{sec:DG}, a generalization of the entropy stable polytropic Euler method into a high-order DGSEM framework is provided. Numerical investigations in Sect.~\ref{sec:numRes} verify the high-order nature of the approximations as well as the entropic properties. Concluding remarks are given in the final section.

\section{Polytropic Euler equations}\label{sec:polyEuler}

We first introduce notation that simplifies the continuous and discrete entropy analysis of the governing equations in this work. The state vector of conserved quantities  is $\statevec{u}$ and the Cartesian fluxes are denoted by $\statevec{f}_1,\,\statevec{f}_2,\,\statevec{f}_3$. As in \cite{bohm2018,Gassner2017}, we define block vector notation with a double arrow
\begin{equation*}
\bigstatevec{f} = \begin{bmatrix}
\statevec{f}_1\\[0.05cm]
\statevec{f}_2\\[0.05cm]
\statevec{f}_3\\[0.05cm]
\end{bmatrix}.
\end{equation*}
The dot product of a spatial vector with a block vector results in a state vector
\begin{equation*} 
\spacevec g\cdot\bigstatevec f  = \sum\limits_{i = 1}^3 {{{ g}_i}{{\statevec f}_i}}. 
\end{equation*}
Thus, the divergence of a block vector is
\begin{equation*}
\spacevec\nabla  \cdot \bigstatevec f = \left(\statevec{f}_1\right)_{\!x} + \left(\statevec{f}_2\right)_{\!y} + \left(\statevec{f}_3\right)_{\!z}.
\end{equation*}
This allows a compact presentation for systems of hyperbolic conservation laws 
\begin{equation}\label{eq:compactConsLaw}
\statevec{u}_t + \spacevec{\nabla}\cdot\bigstatevec{f} = \statevec{0},
\end{equation}
on a domain $\Omega\subset\mathbb{R}^3$.  

\subsection{Governing equations}

The polytropic Euler equations are a simplified version of the compressible Euler equations \eqref{eq:compEuler} which explicitly conserves the mass and momenta. In the equation of state for polytropic fluids the pressure depends solely on the fluid density and the total energy conservation law becomes redundant \cite{chen2005}. The simplified system takes the form of non-linear conservation laws \eqref{eq:compactConsLaw} with
\begin{equation*}
\statevec{u} = \begin{bmatrix}
\rho\\[0.05cm]
\rho\spacevec{v}\\[0.05cm]
\end{bmatrix},
\qquad
\bigstatevec{f} = \begin{bmatrix}
\rho\spacevec{v}\\[0.05cm]
\rho\spacevec{v}\otimes\spacevec{v} + p\threeMatrix{I}\\[0.05cm]
\end{bmatrix},
\end{equation*}
where $\threeMatrix{I}$ is a $3\times 3$ identity matrix. We close the system with the polytropic or the isothermal gas assumption, which relate density and pressure:
\begin{equation}\label{eq:eos}
\text{polytropic case:} \quad p(\rho)=\kappa \rho^\gamma,
\qquad\text{isothermal case:} \quad p(\rho)=c^2 \rho.
\end{equation}
For a polytropic gas $ \gamma > 1$ is the adiabatic coefficient and $ \kappa > 0 $ is some scaling factor depending on the fluid, e.g. for the shallow water equations with $\kappa = g/2$ (gravitational acceleration) and $\gamma = 2$ \cite{Fjordholm2011}. For the isothermal case $\gamma=1$ and $c > 0$ is the speed of sound \cite{chen2005}. To keep the analysis of the polytropic Euler equations general, we will only specify which equation of state is used when necessary. Further, in barotropic models the internal energy, $e(\rho)$, and the pressure form an admissible pair provided the ordinary differential equation
\begin{equation}\label{eq:innerPressure}
\rho\frac{de}{d\rho} = \frac{p(\rho)}{\rho},
\end{equation}
is satisfied \cite{chen2005}.
For the equations of state \eqref{eq:eos} the corresponding internal energies are
\begin{equation}\label{eq:innerEnergy}
\text{polytropic case:} \quad e(\rho)= \frac{\kappa \rho^{\gamma-1}}{\gamma-1},
\qquad\text{isothermal case:} \quad e(\rho)=c^2 \ln(\rho).
\end{equation}

\subsection{Continuous entropy analysis}

We define the necessary components to discuss the thermodynamic properties of \eqref{eq:compactConsLaw} from a mathematical perspective. To do so, we utilize well-developed entropy analysis tools, e.g. \cite{harten1983,mock1980,Tadmor1987}. First, we introduce an entropy function used to define an injective mapping between state space and entropy space \cite{harten1983,mock1980}. 

For the polytropic Euler equations, a suitable mathematical entropy function is the total energy of the system \cite{chen2005}
\begin{equation}\label{eq:totalEnergy}
s(\statevec{u}) = \frac{\rho}{2}\|\spacevec{v}\|^2 + \rho e(\rho),
\end{equation}
with the internal energy taken from \eqref{eq:innerEnergy}. Note that the entropy function $ s(\statevec{u}) $ is strongly convex under the physical assumption that $\rho>0$. From the entropy function we find the entropy variables to be
\begin{equation}\label{eq:entVars}
\statevec{w} = \pderivative{s}{\statevec{u}} = \left(e + \rho\frac{de}{d\rho} - \frac{1}{2}\|\spacevec{v}\|^2,v_1,v_2,v_3\right)^T=\left(e + \frac{p}{\rho} - \frac{1}{2}\|\spacevec{v}\|^2,v_1,v_2,v_3\right)^T\!\!,
\end{equation}
where we use the relation \eqref{eq:innerPressure} to simplify the first entropy variable. The mapping between state space and entropy space is equipped with symmetric positive definite (s.p.d) entropy Jacobian matrices, e.g., \cite{Tadmor1987} 
\begin{equation*}
\nineMatrix{H}^{-1} = \pderivative{\statevec{w}}{\statevec{u}},
\end{equation*}
and
\begin{equation}\label{eq:entJac}
\nineMatrix{H} = \frac{1}{a^2}\begin{bmatrix}
\rho & \rho v_1 & \rho v_2 & \rho v_3\\[0.1cm]
\rho v_1 & \rho v_1^2 + a^2\rho & \rho v_1 v_2 & \rho v_1 v_3\\[0.1cm]
\rho v_2 & \rho v_1 v_2 & \rho v_2^2 + a^2\rho& \rho v_2 v_3 \\[0.1cm]
\rho v_3 & \rho v_1 v_3 & \rho v_2 v_3 & \rho v_3^2 + a^2\rho\\[0.1cm]
\end{bmatrix},
\end{equation}
where we introduce a general notation for the sound speed
\begin{equation*}
a^2 = \frac{\gamma p}{\rho}.
\end{equation*}
We note that this statement of $\nineMatrix{H}$ is general for either equation of state from \eqref{eq:eos}. The entropy fluxes, $\spacevec{f}^{\,s}$, associated with the entropy function \eqref{eq:totalEnergy} are
\begin{equation}\label{eq:entFluxes}
\spacevec{f}^{\,s} = (f_1^{s},f_2^{s},f_3^{s})^T = \spacevec{v} \left(s + p\right).
\end{equation}
Finally, we compute the entropy flux potential that is needed later in Sect.~\ref{sec:ECFlux} for the construction of entropy conservative numerical flux functions
\begin{equation}\label{eq:entPot}
\spacevec{\Psi} = \statevec{w}^T\bigstatevec{f} - \spacevec{f}^{\,s} = \spacevec{v}p.
\end{equation}

To examine the mathematical entropy conservation we contract the system of conservation laws \eqref{eq:compactConsLaw} from the left with the entropy variables \eqref{eq:entVars}. By construction, and assuming continuity, the time derivative term becomes
\begin{equation*}
\statevec{w}^T\pderivative{\statevec{u}}{t} = \pderivative{s}{t}.
\end{equation*}
The contracted flux terms, after many algebraic manipulations, yield
\begin{equation*}
\statevec{w}^T\spacevec{\nabla}\cdot\bigstatevec{f} = \cdots = \spacevec{\nabla}\cdot\left(\frac{\rho\spacevec{v}}{2}\|\spacevec{v}\|^2 + \rho e \spacevec{v}\right)
=
\spacevec{\nabla}\cdot\left(\spacevec{v}\left[s + p \right]\right)
=\spacevec{\nabla}\cdot\spacevec{f}^{\,s}.
\end{equation*}
Therefore, for smooth solutions contracting \eqref{eq:compactConsLaw} into entropy space yields an additional conservation law for the total energy
\begin{equation}\label{eq:entCons}
\statevec{w}^T\left(\statevec{u}_t + \spacevec{\nabla}\cdot\bigstatevec{f}\right) = 0\qquad \Rightarrow\qquad s_t + \spacevec{\nabla}\cdot\spacevec{f}^{\,s} = 0.
\end{equation}
Generally, discontinuous solutions can develop for non-linear hyperbolic systems, regardless of their initial smoothness. In the presence of discontinuities, the mathematical entropy conservation law \eqref{eq:entCons} becomes the entropy inequality \cite{Tadmor1987}
\begin{equation*}
s_t + \spacevec{\nabla}\cdot\spacevec{f}^{\,s} \leq 0.
\end{equation*}

Note, due to the form of the entropy fluxes \eqref{eq:entFluxes} the mathematical entropy conservation law \eqref{eq:entCons} has an identical form to the conservation of total energy from the adiabatic compressible Euler equations \eqref{eq:compEuler}. This reinforces that the total energy becomes an \textit{auxiliary} conserved quantity for polytropic gases.

\subsection{Eigenstructure of the polytropic Euler equations}

To close this section we investigate the eigenstructure of the polytropic Euler equations. We do so to demonstrate the hyperbolic character of the governing equations. Additionally, a detailed description of the eigenvalues and eigenvectors is needed to select a stable explicit time step \cite{courant1967} as well as design operators that selectively add dissipation to the different propagating waves in the system, e.g. \cite{Winters2017}. 

To simplify the eigenstructure discussion of the polytropic Euler system, we limit the investigation to one spatial dimension. This restriction simplifies the analysis and is done without loss of generality, because the spatial directions are decoupled and the polytropic Euler equations are rotationally invariant. To begin we state the one-dimensional form of \eqref{eq:compactConsLaw}
\begin{equation*}
\statevec{u}_t + (\statevec{f}_1)_x = \statevec{0},
\end{equation*}
where we have
\begin{equation*}
\statevec{u} = \left[\rho,\rho v_1,\rho v_2,\rho v_3\right]^T,\qquad
\statevec{f}_1 = \left[\rho v_1\,,\,\rho v_1^2 + p\,,\,\rho v_1 v_2\,,\,\rho v_1 v_3\right]^T.
\end{equation*}
We find the flux Jacobian matrix to be
\begin{equation}\label{eq:fluxJac}
\nineMatrix{A} = \pderivative{\statevec{f}_1}{\statevec{u}} = 
\begin{bmatrix}
0 & 1 & 0 & 0  \\[0.1cm]
a^2 - v_1^2 & 2v_1 & 0 & 0 \\[0.1cm]
-v_1 v_2 & v_2 & v_1 & 0 \\[0.1cm]
-v_1 v_3 & v_3 & 0 & v_1 \\[0.1cm]
\end{bmatrix}.
\end{equation}
The eigenvalues, $\{\lambda_i\}_{i=1}^4$, of \eqref{eq:fluxJac} are all real
\begin{equation}\label{eq:evals}
\lambda_1 = v_1 - a \qquad \lambda_2=v_1 \qquad \lambda_3=v_1 \qquad \lambda_4 = v_1 + a.
\end{equation}
The eigenvalues are associated with a full set of right eigenvectors. A matrix of right eigenvectors is
\begin{equation}\label{eq:rightEvecs}
\nineMatrix{R} = \left[\statevec{r}_1\,|\,\statevec{r}_2\,|\,\statevec{r}_3\,|\,\statevec{r}_4\right] = \begin{bmatrix}
1 & 0 & 0 & 1\\[0.05cm]
v_1 - a & 0 & 0 & v_1 + a\\[0.05cm]
v_2 & 1 & 0 & v_2\\[0.05cm]
v_3 & 0 & 1 & v_3\\[0.05cm]
\end{bmatrix}.
\end{equation}
From the work of Barth \cite{Barth1999}, there exists a positive diagonal scaling matrix $\nineMatrix{Z}$ that relates the right eigenvector matrix \eqref{eq:rightEvecs} to the entropy Jacobian matrix \eqref{eq:entJac}
\begin{equation}\label{eq:evEntScaling}
\nineMatrix{H} = \nineMatrix{R}\nineMatrix{Z}\nineMatrix{R}^T.
\end{equation}
For the polytropic Euler equations this diagonal scaling matrix is
\begin{equation*}
\nineMatrix{Z} =  \mathrm{diag}\left(\frac{\rho}{2a^2}\,,\,\rho\,,\,\rho\,,\,\frac{\rho}{2a^2}\right). 
\end{equation*}
We will revisit the eigenstructure of the polytropic Euler equations and this eigenvector scaling in Sect.~\ref{sec:ESFlux} in order to derive an entropy stable numerical dissipation term.

\section{Discrete entropy analysis, finite volume, and numerical fluxes}\label{sec:discEntropy}

In this section we derive entropy conservative and entropy stable numerical flux functions for the polytropic Euler equations. This discrete analysis is performed in the context of finite volume schemes and follows closely the work of Tadmor \cite{Tadmor2003}. The derivations for entropy conservative numerical flux functions and appropriate dissipation terms are straightforward, albeit algebraically involved. Therefore, we restrict the discussion to the one dimensional version of the model for the sake of simplicity. As such, we suppress the subscript on the physical flux and simply state $\statevec{f}$.

\subsection{Finite volume discretization}

Finite volume methods are a discretization technique particularly useful to approximate the solution of hyperbolic systems of conservation laws. The method is developed from the integral form of the equations \cite{leveque2002}
\begin{equation*}
\int\limits_{\Omega} \statevec{u}_t\,\mathrm{d}\spacevec{x} + \int\limits_{\partial\Omega}\statevec{f}\cdot\spacevec{n}\,\mathrm{d}S = 0,
\end{equation*}
where $\spacevec{n}$ is the outward pointing normal vector. In one spatial dimension we divide the interval into non-overlapping cells
\begin{equation*}
\Omega_i = \left[x_{i-\tfrac{1}{2}},x_{i+\tfrac{1}{2}}\right],
\end{equation*}
and the integral equation contributes
\begin{equation*}
\frac{d}{dt}\int\limits_{x_{i-1/2}}^{x_{i+1/2}}\statevec{u}\,\mathrm{d}x + \statevec{f}^*(x_{i+1/2}) - \statevec{f}^*(x_{i-1/2})  = 0,
\end{equation*}
on each cell.
The solution approximation is assumed to be a constant value within the volume. Then we determine the cell average value with, for example, a midpoint quadrature of the solution integral
\begin{equation*}
\int\limits_{x_{i-1/2}}^{x_{i+1/2}}\statevec{u}\,\mathrm{d}x \approx \int\limits_{x_{i-1/2}}^{x_{i+1/2}}\statevec{u}_i\,\mathrm{d}x = \statevec{u}_i\Delta x_i.
\end{equation*}
Due to the integral form of the finite volume scheme the solution is allowed to be discontinuous at the boundaries of the cells. To resolve this, we introduce a numerical flux, $\statevec{f}^*\left(\statevec{u}_L,\statevec{u}_R\right)$ \cite{leveque2002,toro2009} which is a function of two solution states at a cell interface and returns a single flux value. For consistency, we require that
\begin{equation}\label{eq:consistencyFV}
\statevec{f}^*\left(\statevec{q},\statevec{q}\right) = \statevec{f},
\end{equation}
such that the numerical flux is equivalent to the physical flux when evaluated at two identical states.

The resulting finite volume spatial approximation takes the general form
\begin{equation}\label{eq:finalFiniteVolume}
\left(\statevec{u}_t\right)_i + \frac{1}{\Delta x_i}\left(\statevec{f}^*_{x+\tfrac{1}{2}} - \statevec{f}^*_{x-\tfrac{1}{2}}\right) = \statevec{0}.
\end{equation}
This results in a set of temporal ordinary differential equations that can be integrated with an appropriate ODE solver, e.g., explicit Runge-Kutta methods.

To complete the spatial approximation \eqref{eq:finalFiniteVolume} requires a suitable numerical flux function $\statevec{f}^*$. Next, following the work of Tadmor, we will develop entropy conservative and entropy stable numerical fluxes for the polytropic Euler equations.

\subsection{Entropy conservative numerical flux}\label{sec:ECFlux}

First, we develop the entropy conservative flux function $\statevec{f}^{*,\ec}$ that is valid for smooth solutions and acts as the baseline for the entropy stable numerical approximation. We assume left and right cell averages, denoted by $L$ and $R$, on uniform cells of size $\Delta x$ separated by a common interface. We discretize the one-dimensional system semi-discretely and derive an approximation for the fluxes at the interface between the two cells, i.e., at the $i+1/2$ interface:
\begin{equation}\label{eq:twoStates}
\Delta x \pderivative{\statevec{u_L}}{t} = \statevec{f}_L - \statevec{f}^*\qquad\mathrm{and}\qquad\Delta x \pderivative{\statevec{u_R}}{t} = \statevec{f}^* - \statevec{f}_R,
\end{equation}
where the adjacent states feature the physical fluxes $\statevec{f}_{L,R}$ and the numerical interface flux $\statevec{f}^*$. We define the jump in a quantity across an interface by
\begin{equation*}
\jump{\cdot} = (\cdot)_R - (\cdot)_L.
\end{equation*}
Next, we contract \eqref{eq:twoStates} into entropy space to obtain the semi-discrete entropy update in each cell
\begin{equation}\label{eq:twoStates2}
\Delta x\pderivative{s_L}{t} = \statevec{w}_L^T\left(\statevec{f}_L - \statevec{f}^*\right)\qquad\mathrm{and}\qquad \Delta x\pderivative{s_R}{t} = \statevec{w}_R^T\left(\statevec{f}^* - \statevec{f}_R\right),
\end{equation}
where we assume continuity in time such that $s_t = \statevec{w}^T\statevec{u}_t$.

Next, we add the contributions from each side of the interface in \eqref{eq:twoStates2} to obtain the total entropy update
\begin{equation}\label{eq:discreteEntUpdate1}
\Delta x\pderivative{}{t}\left(s_L+s_R\right) = \jump{\statevec{w}}^T\statevec{f}^* - \jump{\statevec{w}^T\statevec{f}}.
\end{equation}
To ensure that the finite volume update satisfies the discrete entropy conservation law, the entropy flux of the finite volume discretization must coincide with the discrete entropy flux $f^{s}$, i.e.,
\begin{equation*}
 \jump{\statevec{w}}^T\statevec{f}^* - \jump{\statevec{w}^T\statevec{f}} \stackrel[]{!}{=}-\jump{f^s}.
\end{equation*}
We use the linearity of the jump operator and rearrange to obtain the general entropy conservation condition of Tadmor \cite{Tadmor1987}
\begin{equation}\label{eq:cond1}
\jump{\statevec{w}}^T\statevec{f}^* = \jump{\statevec{w}^T\statevec{f} - f^s} = \jump{\Psi},
\end{equation}
where we apply the definition of the entropy flux potential \eqref{eq:entPot}. The discrete entropy conservation condition \eqref{eq:cond1} is a single constraint for a vector quantity. Thus, the form of the entropy conservative numerical flux is not unique. However, the resulting numerical flux form \eqref{eq:cond1} must remain consistent \eqref{eq:consistencyFV}.

To derive an entropy conservative flux we note the properties of the jump operator
\begin{equation}\label{eq:jumpProps}
\jump{ab} = \avg{a}\jump{b} + \avg{b}\jump{a}, \quad \jump{a^2} = 2\avg{a}\jump{a},
\end{equation}
where we introduce notation for the arithmetic mean
\begin{equation*}
\avg{\cdot} = \frac{1}{2}\left((\cdot)_R+(\cdot)_L\right).
\end{equation*}
For the numerical flux to remain applicable to either equation of state \eqref{eq:eos} we require a generalized average operator for the fluid density.
\begin{definition}[$\gamma$-mean]\label{def:gamAvg}
Assuming that $\rho_L\ne\rho_R$, a special average of the fluid density is
\begin{equation}\label{eq:gammaAvg}
\gamavg = \frac{1}{\gamma}\frac{\jump{p}}{\jump{e}}.
\end{equation}
We examine the evaluation of the average \eqref{eq:gammaAvg} for three cases substituting the appropriate forms of the pressure \eqref{eq:eos} and internal energy \eqref{eq:innerEnergy}:
\begin{enumerate}
\item Polytropic ($\gamma>1$) yields
\begin{equation}\label{eq:gam1}
\gamavg = \frac{1}{\gamma}\frac{\jump{\kappa\rho^\gamma}}{\jump{\frac{\kappa}{\gamma-1}\rho^{\gamma-1}}} = \frac{\gamma-1}{\gamma}\frac{\jump{\rho^\gamma}}{\jump{\rho^{\gamma-1}}}.
\end{equation}
\item Isothermal ($\gamma=1$) where the special average becomes the logarithmic mean which also arises in the construction of entropy conservative fluxes for the adiabatic Euler equations \cite{Chandrashekar2012,IsmailRoe2009}
\begin{equation}\label{eq:logMean}
\gamavg = \frac{1}{1}\frac{\jump{c^2\rho}}{\jump{c^2\ln(\rho)}} = \frac{\jump{\rho}}{\jump{\ln(\rho)}} \coloneqq \rho^{\ln}.
\end{equation}
\item Shallow water ($\gamma=2$) for which the special average reduces to the arithmetic mean
\begin{equation*}
\gamavg = \frac{1}{2}\frac{\jump{\rho^2}}{\jump{\rho}} = \frac{1}{2}\frac{2\avg{\rho}\jump{\rho}}{\jump{\rho}} = \avg{\rho}.
\end{equation*}
\end{enumerate}
\end{definition}
\begin{remark}
The $\gamma$-mean \eqref{eq:gam1} is a special case of the weighted Stolarsky mean, which serves as a generalization of the logarithmic mean \cite{stolarsky1975}. It remains consistent when the left and right states are identical. Also, assuming without loss of generality that $\rho_L<\rho_R$, it is guaranteed that the value of $\gamavg\in[\rho_L,\rho_R]$ \cite{stolarsky1975,stolarsky1980}.
\end{remark}
\begin{remark}
In practice, when the left and right fluid density values are close, there are numerical stability issues because the $\gamma$-mean \eqref{eq:gammaAvg} tends to a $0/0$ form. Therefore, we provide a numerically stable procedure to compute \eqref{eq:gammaAvg} in Appendix \ref{app:gamAvgEval}.
\end{remark}

With Def. \ref{def:gamAvg} and the discrete entropy conservation condition \eqref{eq:cond1} we are equipped to derive an entropy conservative numerical flux function.
\begin{theorem}[Entropy conservative flux]\label{thm:EC}
From the discrete entropy conservation condition condition \eqref{eq:cond1} we find a consistent, entropy conservative numerical flux
\begin{equation}\label{eq:ECflux}
\statevec{f}^{*,\ec} = \begin{bmatrix}
\gamavg\avg{v_1}\\[0.1cm]
\gamavg\avg{v_1}^2 + \avg{p}\\[0.1cm]
\gamavg\avg{v_1}\avg{v_2}\\[0.1cm]
\gamavg\avg{v_1}\avg{v_3}\\[0.1cm]
\end{bmatrix}.
\end{equation}
\end{theorem}
\begin{proof}
We first expand the right-hand-side of the entropy conservation condition \eqref{eq:cond1}
\begin{equation}\label{eq:cond33}
\jump{\Psi} = \jump{pv_1} = \avg{v_1}\jump{p}+ \avg{p}\jump{v_1},
\end{equation}
where we use the jump properties \eqref{eq:jumpProps}. Next, we expand the jump in the entropy variables. To do so, we revisit the form of the entropy variable $w_1$ as it changes depending on the equation of state
\begin{equation}\label{eq:entVarNumber1}
w_1 = \begin{cases}
e +\frac{\kappa\rho^\gamma}{\rho} -\frac{1}{2}\|\spacevec{v}\|^2 = \frac{\kappa\rho^{\gamma-1}}{\gamma-1} + \kappa\rho^{\gamma-1} -\frac{1}{2}\|\spacevec{v}\|^2 = \gamma e - \frac{1}{2}\|\spacevec{v}\|^2, & \mathrm{polytropic}\\[0.1cm]
e + \frac{c^2\rho}{\rho}-\frac{1}{2}\|\spacevec{v}\|^2 = e + c^2 -\frac{1}{2}\|\spacevec{v}\|^2, & \mathrm{isothermal}
\end{cases}.
\end{equation}
Taking the jump of the variable \eqref{eq:entVarNumber1} we obtain
\begin{equation}\label{eq:entVarJump1}
\jump{w_1} = \begin{cases}
\gamma \jump{e} - \avg{v_1}\jump{v_1} - \avg{v_2}\jump{v_2} - \avg{v_3}\jump{v_3}, & \mathrm{polytropic}\\[0.1cm]
\jump{e} - \avg{v_1}\jump{v_1} - \avg{v_2}\jump{v_2} - \avg{v_3}\jump{v_3}, & \mathrm{isothermal}
\end{cases}
\end{equation}
because the values of $\gamma$ and $c$ are constant. Note that in the isothermal case $\gamma = 1$, so the jump of $w_1$ \eqref{eq:entVarJump1} has the same form regardless of the equation of state. Therefore, the total jump in the entropy variables is
\begin{equation}\label{eq:jumpEntVars}
\jump{\statevec{w}}^T = \begin{bmatrix}
 \jump{e + \frac{p}{\rho} - \frac{1}{2}\|\spacevec{v}\|^2}\\[0.1cm]
 \jump{v_1}\\[0.1cm]
 \jump{v_2}\\[0.1cm]
 \jump{v_3}\\[0.1cm]
\end{bmatrix}
 = \begin{bmatrix}
 \gamma\jump{e} - \avg{v_1}\jump{v_1} - \avg{v_2}\jump{v_2} - \avg{v_3}\jump{v_3}\\[0.1cm]
 \jump{v_1}\\[0.1cm]
 \jump{v_2}\\[0.1cm]
 \jump{v_3}\\[0.1cm]
\end{bmatrix}.
\end{equation}
We combine the expanded condition \eqref{eq:cond33}, the jump in the entropy variables \eqref{eq:jumpEntVars}, and rearrange terms to find
\begin{equation}\label{eq:cond5}
\begin{aligned}
f_1^*\left(\gamma\jump{e}-\avg{v_1}\jump{v_1}-\avg{v_2}\jump{v_2}-\avg{v_3}\jump{v_3}\right)&+f_2^*\jump{v_1}+f_3^*\jump{v_2}+f_4^*\jump{v_3}\\[0.1cm]
&= \avg{v_1}\jump{p}+\avg{p}\jump{v_1}.
\end{aligned}
\end{equation}

To determine the first flux component we find
\begin{equation}\label{eq:flux1}
f_1^*\gamma\jump{e} = \avg{v_1}\jump{p}\qquad \Rightarrow\qquad f_1^* = \frac{1}{\gamma}\frac{\jump{p}}{\jump{e}}\avg{v_1} = \gamavg\avg{v_1},
\end{equation}
from the $\gamma$-mean in Def. \ref{def:gamAvg}. The expanded flux condition \eqref{eq:cond5} is rewritten into linear jump components. We gather the like terms of each jump component to facilitate the construction of the remaining flux components:
\begin{equation}\label{eq:Jump1}
\begin{aligned}
\jump{v_1}:\quad -f_1^*\avg{v_1} + f_2^* &= \avg{p},\\[0.05cm]
\jump{v_2}:\quad -f_1^*\avg{v_2} + f_3^* &= 0,\\[0.05cm]
\jump{v_3}:\quad -f_1^*\avg{v_3} + f_4^* &= 0.
\end{aligned}
\end{equation}
Now, it is straightforward to solve the expressions in \eqref{eq:Jump1} and find
\begin{equation}\label{eq:solve1}
\begin{aligned}
f_2^* &= \gamavg\avg{v_1}^2 + \avg{p},\\[0.05cm]
f_3^* &= \gamavg\avg{v_1}\avg{v_2},\\[0.05cm]
f_4^* &= \gamavg\avg{v_1}\avg{v_3}.
\end{aligned}
\end{equation}
If we assume the left and right states are identical in \eqref{eq:flux1} and \eqref{eq:solve1} it is straightforward to verify that the numerical flux is consistent from its form and the properties of the $\gamma$-mean.\qed
\end{proof}
\begin{remark}
There are two values of $\gamma$ that change the form of the entropy conservative flux \eqref{eq:ECflux}:
\begin{enumerate}
\item Isothermal case ($\gamma=1$): The numerical flux becomes
\begin{equation*}
\statevec{f}^{*,\ec}_{\mathrm{iso}} = \begin{bmatrix}
\rho^{\ln}\avg{v_1}\\[0.1cm]
\rho^{\ln}\avg{v_1}^2 + \avg{p}\\[0.1cm]
\rho^{\ln}\avg{v_1}\avg{v_2}\\[0.1cm]
\rho^{\ln}\avg{v_1}\avg{v_3}\\[0.1cm]
\end{bmatrix},
\end{equation*}
where the fluid density is computed with the logarithmic mean \eqref{eq:logMean} just as in the adiabatic case \cite{Chandrashekar2012,IsmailRoe2009}.
\item Shallow water case ($\gamma=2$): Here the numerical flux simplifies to become
\begin{equation*}
\statevec{f}^{*,\ec}_{\mathrm{sw}} = \begin{bmatrix}
\avg{\rho}\avg{v_1}\\[0.1cm]
\avg{\rho}\avg{v_1}^2 + \avg{p}\\[0.1cm]
\avg{\rho}\avg{v_1}\avg{v_2}\\[0.1cm]
\end{bmatrix} 
= 
\begin{bmatrix}
\avg{\rho}\avg{v_1}\\[0.1cm]
\avg{\rho}\avg{v_1}^2 + \kappa\avg{\rho^2}\\[0.1cm]
\avg{\rho}\avg{v_1}\avg{v_2}\\[0.1cm]
\end{bmatrix},
\end{equation*}
where the velocity component in the $z$ direction is ignored due to the assumptions of the shallow water equations \cite{whitham1974}. If we let the fluid density be denoted as the water height, $\rho\mapsto h$, and take $\kappa=g/2$ where $g$ is the gravitational constant, then we recover the entropy conservative numerical flux function originally developed for the shallow water equations by Fjordholm et al. \cite{Fjordholm2011}.
\end{enumerate}
\end{remark}
\begin{remark}[Multi-dimensional entropy conservative fluxes]
The derivation of entropy conservative numerical fluxes in the other spatial directions is very similar to that shown in Thm. \ref{thm:EC}. So, we present the three-dimensional entropy conservative fluxes in the $y$ and $z$ directions in Appendix \ref{app:3DFlux}.
\end{remark}

\subsection{Entropy stable numerical flux}\label{sec:ESFlux}

As previously mentioned, the solution of hyperbolic conservation laws can contain or develop discontinuities regardless of the smoothness of the initial conditions \cite{evans2010}. In this case, a numerical approximation that is entropy conservative is no longer physical and should account for the dissipation of entropy near discontinuities. Such a numerical method is deemed entropy stable, e.g., \cite{fisher2013_2,Tadmor1987}. To create an entropy stable numerical flux function we begin with a general form
\begin{equation}\label{eq:ESFlux}
\statevec{f}^{*,\es} = \statevec{f}^{*,\ec} - \frac{1}{2}\nineMatrix{D}\jump{\statevec{u}},
\end{equation}
where $\nineMatrix{D}$ is a symmetric positive definite dissipation matrix. An immediate issue arises when we contract the entropy stable flux \eqref{eq:ESFlux} into entropy space. We must guarantee that the dissipation term possesses the correct sign \cite{Winters2016}; however, contracting \eqref{eq:ESFlux} with the jump in entropy variables gives
\begin{equation}\label{eq:entStable1}
\begin{aligned}
\jump{\statevec{w}}^T\statevec{f}^{*,\es} &= \jump{\statevec{w}}^T\statevec{f}^{*,\ec} - \frac{1}{2}\jump{\statevec{w}}^T\nineMatrix{D}\jump{\statevec{u}}\\[0.1cm]
&=-\jump{f^{s}} -\frac{1}{2}\jump{\statevec{w}}^T\nineMatrix{D}\jump{\statevec{u}}.
\end{aligned}
\end{equation}
So, there is a mixture of entropy and conservative variable jumps in the dissipation term that must be guaranteed positive to ensure that entropy is dissipated correctly. In general, it is unclear how to guarantee positivity of the dissipation term in \eqref{eq:entStable1} as required for entropy stability \cite{Barth1999}. To remedy this issue we rewrite $\jump{\statevec{u}}$ in terms of $\jump{\statevec{w}}$. This is possible due to the one-to-one variable mapping between conservative and entropy space as we know that
\begin{equation*}
\pderivative{\statevec{u}}{x} = \nineMatrix{H}\pderivative{\statevec{w}}{x}.
\end{equation*}

For the discrete case we wish to recover a particular average evaluation of the entropy Jacobian \eqref{eq:entJac} at a cell interface such that
\begin{equation}\label{eq:Hrelation}
\jump{\statevec{u}} \stackrel[]{!}{=} \hat{\nineMatrix{H}}\jump{\statevec{w}}.
\end{equation}
To generate a discrete entropy Jacobian that satisfies \eqref{eq:Hrelation} we need a specially designed average for the square of the sound speed.
\begin{definition}[Average square sound speed]
A special average for the sound speed squared is
\begin{equation}\label{eq:soundSpeedAvg}
\aAvg = \frac{\jump{p}}{\jump{\rho}}.
\end{equation}
\end{definition}
\begin{remark}
The average \eqref{eq:soundSpeedAvg} is consistent. To demonstrate this, consider the polytropic equation of state from \eqref{eq:eos} and take $\rho_R = \rho + \epsilon$ and $\rho_L=\rho$ so that 
\begin{equation*}
\aAvg = \frac{\jump{p}}{\jump{\rho}} = \frac{\frac{1}{\epsilon}\kappa\left((\rho+\epsilon)^{\gamma} - \rho^{\gamma}\right)}{\frac{1}{\epsilon}\left((\rho+\epsilon) - \rho\right)} \stackrel[]{\epsilon\rightarrow 0}{=} \gamma \kappa\rho^{\gamma-1} = \frac{\gamma \kappa\rho^{\gamma}}{\rho} = \frac{\gamma p}{\rho} = a^2.
\end{equation*}
\end{remark}
\begin{remark}\label{rem:soundSpeedAvg}
Again examining the special values of $\gamma$ we find:
\begin{enumerate}
\item Isothermal ($\gamma=1$) gives
\begin{equation*}
\aAvg = \frac{\jump{p}}{\jump{\rho}} = \frac{\jump{c^2\rho}}{\jump{\rho}} = c^2,
\end{equation*}
as $c^2$ is a constant.
\item Shallow water ($\gamma=2$) yields
\begin{equation*}
\aAvg = \frac{\jump{p}}{\jump{\rho}} = \frac{\jump{\kappa\rho^2}}{\jump{\rho}} = \frac{g}{2}\frac{2\avg{\rho}\jump{\rho}}{\jump{\rho}} = g\!\avg{\rho},
\end{equation*}
where, again, we take $\kappa=g/2$ and apply a property of the jump operator \eqref{eq:jumpProps}. Denoting the fluid density as the water height, $\rho\mapsto h$, we recover an average of the wave celerity for the shallow water model \cite{Fjordholm2011}.
\end{enumerate}
\end{remark}
\begin{remark}
Just as with the $\gamma$-mean, the sound speed average \eqref{eq:soundSpeedAvg} exhibits numerical stability issues for the polytropic case $\gamma>1$ when the fluid density values are close. Therefore, we present a numerically stable procedure to evaluate \eqref{eq:soundSpeedAvg} in Appendix \ref{app:aAvgEval}.
\end{remark}
\begin{lemma}[Discrete entropy Jacobian evaluation]
If the entropy Jacobian is evaluated with the average states
\begin{equation}\label{eq:discreteHMatrix}
\hat{\nineMatrix{H}} = \frac{1}{\aAvg}\begin{bmatrix}
\gamavg & \gamavg\avg{v_1} & \gamavg\!\avg{v_2} & \gamavg\!\avg{v_3}\\[0.1cm]
\gamavg\!\avg{v_1} & \gamavg\!\avg{v_1}^2 + \aAvg\avg{\rho} & \gamavg\!\avg{v_1}\avg{v_2} & \gamavg\!\avg{v_1}\avg{v_3}\\[0.1cm]
\gamavg\!\avg{v_2} & \gamavg\!\avg{v_1}\avg{v_2} & \gamavg\!\avg{v_2}^2 +  \aAvg\avg{\rho} & \gamavg\!\avg{v_2}\avg{v_3}\\[0.1cm]
\gamavg\!\avg{v_3} & \gamavg\!\avg{v_1}\avg{v_3} & \gamavg\!\avg{v_2}\avg{v_3} & \gamavg\!\avg{v_3}^2 +  \aAvg\avg{\rho}\\[0.1cm]
\end{bmatrix}
\end{equation}
then it is possible to relate the jump in conservative variables in terms of the jump in entropy variables by
\begin{equation}\label{eq:matrixConditionLemma}
\jump{\statevec{u}} = \hat{\nineMatrix{H}}\jump{\statevec{w}}.
\end{equation}
\end{lemma}
\begin{proof}
We demonstrate how to obtain the first row of the discrete matrix $\hat{\nineMatrix{H}}$. From the condition \eqref{eq:Hrelation} we see that
\begin{equation}\label{eq:HFirstRow}
\jump{\rho} \stackrel[]{!}{=} \hat{\nineMatrix{H}}_{11}\left(\gamma\jump{e}-\avg{v_1}\jump{v_1}-\avg{v_2}\jump{v_2}-\avg{v_3}\jump{v_3}\right) + \hat{\nineMatrix{H}}_{12}\jump{v_1}+ \hat{\nineMatrix{H}}_{13}\jump{v_2}+ \hat{\nineMatrix{H}}_{14}\jump{v_3}.
\end{equation}
To determine the first entry of the matrix we apply the definition of the $\gamma$-mean \eqref{eq:gammaAvg} and the sound speed average \eqref{eq:soundSpeedAvg} to obtain
\begin{equation*}\label{eq:H11}
\hat{\nineMatrix{H}}_{11} = \frac{1}{\gamma}\frac{\jump{\rho}}{\jump{e}} = \frac{1}{\gamma}\frac{\jump{\rho}}{\frac{1}{\gamma}\frac{\jump{p}}{\gamavg}} = \gamavg\frac{\jump{\rho}}{\jump{p}} = \frac{\gamavg}{\aAvg}.
\end{equation*}
The remaining components in the first row of $\hat{\nineMatrix{H}}$ from \eqref{eq:HFirstRow} are
\begin{equation*}
\hat{\nineMatrix{H}}_{12} = \frac{\gamavg\avg{v_1}}{\aAvg},\quad \hat{\nineMatrix{H}}_{13} = \frac{\gamavg\avg{v_2}}{\aAvg},\quad\hat{\nineMatrix{H}}_{14} = \frac{\gamavg\avg{v_3}}{\aAvg}.
\end{equation*}
Repeating this process we obtain the remaining unknown components in the relation \eqref{eq:Hrelation} and arrive at the discrete entropy Jacobian \eqref{eq:discreteHMatrix}.\qed
\end{proof}

Next, we select the dissipation matrix $\nineMatrix{D}$ to be a discrete evaluation of the eigendecomposition of the flux Jacobian \eqref{eq:fluxJac}
\begin{equation}\label{eq:dissMat1}
\nineMatrix{D} = \hat{\nineMatrix{R}}|\hat{\nineMatrix{\Lambda}}|\hat{\nineMatrix{R}}^{-1},
\end{equation}
where $\nineMatrix{R}$ is the matrix of right eigenvectors \eqref{eq:rightEvecs} and $\nineMatrix{\Lambda}$ is a diagonal matrix containing the eigenvalues \eqref{eq:evals}. From the discrete entropy Jacobian \eqref{eq:discreteHMatrix}, we seek a right eigenvector and diagonal scaling matrix $\hat{\nineMatrix{Z}}$ such that a discrete version of the eigenvector scaling \eqref{eq:evEntScaling}
\begin{equation}\label{eq:matrixCondition}
\hat{\nineMatrix{H}} \stackrel[]{!}{=} \hat{\nineMatrix{R}}\hat{\nineMatrix{Z}}\hat{\nineMatrix{R}}^T,
\end{equation}
holds \textit{whenever possible}.
\begin{lemma}[Discrete eigenvector and scaling matrices]\label{lem:RZ}
If we evaluate the right eigenvector and diagonal scaling matrices as
\begin{equation}\label{eq:discreteRandZ}
\hat{\nineMatrix{R}} = \begin{bmatrix}
1 & 0 & 0 & 1\\[0.1cm]
\avg{v_1} - \sqrt{\aAvg} & 0 & 0 & \avg{v_1} + \sqrt{\aAvg}\\[0.1cm]
\avg{v_2} & 1 & 0 & \avg{v_2}\\[0.1cm]
\avg{v_3} & 0 & 1 & \avg{v_3}\\[0.1cm]
\end{bmatrix},\quad
\hat{\nineMatrix{Z}} = \mathrm{diag}\left(\frac{\gamavg}{2\aAvg}\,,\,\gamavg\,,\,\gamavg\,,\,\frac{\gamavg}{2\aAvg}\right),
\end{equation}
then we obtain the relation
\begin{equation}\label{eq:simeqLemma}
\hat{\nineMatrix{H}}\simeq \hat{\nineMatrix{R}}\hat{\nineMatrix{Z}}\hat{\nineMatrix{R}}^T,
\end{equation}
where equality holds everywhere except for the second, third, and fourth diagonal entries. 
\end{lemma}
\begin{proof}
The procedure to determine the discrete evaluation of the matrices $\hat{\nineMatrix{R}}$ and $\hat{\nineMatrix{Z}}$ is similar to that taken by Winters et al. \cite{Winters2017}. We relate the individual entries of $\hat{\nineMatrix{H}}$ to those in $\hat{\nineMatrix{R}}\hat{\nineMatrix{Z}}\hat{\nineMatrix{R}}^T$ and determine the 16 individual components of the matrices. We explicitly demonstrate two computations to outline the general technique and qualify the average states inserted in the final form. 

We begin by computing the first entry of the first row of the system that should satisfy
\begin{equation*}
\hat{\nineMatrix{H}}_{11} = \frac{\gamavg}{\aAvg} \stackrel[]{!}{=} \frac{\hat{\rho}}{\hat{a}^2} = \left(\hat{\nineMatrix{R}}\hat{\nineMatrix{Z}}\hat{\nineMatrix{R}}^T\right)_{\!11}.
\end{equation*}
This leads to two entries of the diagonal scaling matrix
\begin{equation*}
\hat{\nineMatrix{Z}}_{11} = \hat{\nineMatrix{Z}}_{44} = \frac{\gamavg}{2\aAvg}.
\end{equation*}
The second computation is to determine the second entry of the second row of the system given by
\begin{equation*}
\hat{\nineMatrix{H}}_{22} = \frac{\gamavg\avg{v_1}^2}{\aAvg} + \avg{\rho} \stackrel[]{!}{=} \frac{\gamavg}{2\aAvg}\left(\left(\hat{v}_1+\hat{a}\right)^2 + \left(\hat{v}_1 - \hat{a}\right)^2\right) = \left(\hat{\nineMatrix{R}}\hat{\nineMatrix{Z}}\hat{\nineMatrix{R}}^T\right)_{\!22}.
\end{equation*}
It is clear that we must select $\hat{v}_1 = \avg{v_1}$ and $\hat{a} = \sqrt{\aAvg}$ in the second row of the right eigenvector matrix $\hat{\nineMatrix{R}}$. Unfortunately, just as in the ideal MHD case, we cannot enforce strict equality between the continuous and the discrete entropy scaling analysis \cite{Derigs2016_2,Winters2017} and find
\begin{equation*}
\hat{\nineMatrix{H}}_{22}\simeq\left(\hat{\nineMatrix{R}}\hat{\nineMatrix{Z}}\hat{\nineMatrix{R}}^T\right)_{\!22} = \frac{\gamavg\avg{v_1}^2}{\aAvg} + \gamavg.
\end{equation*}
We apply this same process to the remaining unknown portions from the condition \eqref{eq:matrixCondition} and, after many algebraic manipulations, determine a unique averaging procedure for the discrete eigenvector and scaling matrices \eqref{eq:discreteRandZ}. The derivations were aided and verified using the symbolic algebra software Maxima \cite{maxima}. \qed
\end{proof}
\begin{remark}
In a similar fashion from \cite{Winters2017} we determine the discrete diagonal matrix of eigenvalues
\begin{equation*}
\hat{\nineMatrix{\Lambda}} = \mathrm{diag}\left(\avg{v_1}-\sqrt{\aAvg}\,,\,\avg{v_1}\,,\,\avg{v_1}\,,\,\avg{v_1}+\sqrt{\aAvg}\right).
\end{equation*}
\end{remark}

Now, we have a complete discrete description of the entropy stable numerical flux function from \eqref{eq:ESFlux}
\begin{equation*}
\begin{aligned}
\statevec{f}^{*,\es} &= \statevec{f}^{*,\ec} - \frac{1}{2}\nineMatrix{D}\jump{\statevec{u}}\\[0.1cm]
&=\statevec{f}^{*,\ec} - \frac{1}{2}\hat{\nineMatrix{R}}|\hat{\nineMatrix{\Lambda}}|\hat{\nineMatrix{R}}^{-1}\hat{\nineMatrix{H}}\jump{\statevec{w}},\quad\text{from \eqref{eq:matrixConditionLemma} and \eqref{eq:dissMat1}}\\[0.1cm]
&\simeq\statevec{f}^{*,\ec} - \frac{1}{2}\hat{\nineMatrix{R}}|\hat{\nineMatrix{\Lambda}}|\hat{\nineMatrix{R}}^{-1}\hat{\nineMatrix{R}}\hat{\nineMatrix{Z}}\hat{\nineMatrix{R}}^T\jump{\statevec{w}},\quad\text{from \eqref{eq:simeqLemma}}\\[0.1cm]
&=\statevec{f}^{*,\ec} - \frac{1}{2}\hat{\nineMatrix{R}}|\hat{\nineMatrix{\Lambda}}|\hat{\nineMatrix{Z}}\hat{\nineMatrix{R}}^T\jump{\statevec{w}}.
\end{aligned}
\end{equation*}
This leads to the main result of this section.
\begin{theorem}[Entropy stable flux]\label{thm:ES}
If we select the numerical flux function
\begin{equation}\label{eq:ESFluxInThm}
\statevec{f}^{*,\es} = \statevec{f}^{*,\ec} - \frac{1}{2}\hat{\nineMatrix{R}}|\hat{\nineMatrix{\Lambda}}|\hat{\nineMatrix{Z}}\hat{\nineMatrix{R}}^T\!\jump{\statevec{w}},
\end{equation}
in the discrete entropy update \eqref{eq:discreteEntUpdate1} then the method is guaranteed to dissipate entropy with the correct sign.
\end{theorem}
\begin{proof}
To begin we restate the discrete evolution of the entropy at a single interface where we insert the newly derived entropy stable flux \eqref{eq:ESFluxInThm}
\begin{equation*}\label{eq:EScondInProof}
\Delta x\pderivative{}{t}\left(S_L+S_R\right) = \jump{\statevec{w}}^T\statevec{f}^{*,\es} - \jump{\statevec{w}^T\statevec{f}}.
\end{equation*}
From the construction of the entropy conservative flux function from Thm. \ref{thm:EC} we know that
\begin{equation*}
\jump{\statevec{w}}^T\statevec{f}^{*,\es} - \jump{\statevec{w}^T\statevec{f}} = \jump{\statevec{w}}^T\left(\statevec{f}^{*,\ec} - \frac{1}{2}\hat{\nineMatrix{R}}|\hat{\nineMatrix{\Lambda}}|\hat{\nineMatrix{Z}}\hat{\nineMatrix{R}}^T\jump{\statevec{w}}\right) - \jump{\statevec{w}^T\statevec{f}}=-\jump{f^{s}} - \frac{1}{2}\jump{\statevec{w}}^T\hat{\nineMatrix{R}}|\hat{\nineMatrix{\Lambda}}|\hat{\nineMatrix{Z}}\hat{\nineMatrix{R}}^T\jump{\statevec{w}}.
\end{equation*}
So, \eqref{eq:EScondInProof} becomes
\begin{equation*}
\Delta x\pderivative{}{t}\left(S_L+S_R\right) + \jump{f^{s}} = - \frac{1}{2}\jump{\statevec{w}}^T\hat{\nineMatrix{R}}|\hat{\nineMatrix{\Lambda}}|\hat{\nineMatrix{Z}}\hat{\nineMatrix{R}}^T\jump{\statevec{w}}\leq 0,
\end{equation*}
because the matrices $|\hat{\nineMatrix{\Lambda}}|$ and $\hat{\nineMatrix{Z}}$ are symmetric positive definite, the dissipation term is a quadratic form in entropy space and guarantees a negative contribution.
\qed
\end{proof}

\section{Extension of ES scheme to discontinuous Galerkin numerical approximation}\label{sec:DG}

In this section, we extend the entropy conservative/stable finite volume numerical fluxes to higher spatial order by building them into a nodal discontinuous Galerkin (DG) spectral element method. We provide an \textit{abbreviated} presentation of the entropy stable DG framework, but complete details can be found in \cite{Gassner:2016ye}. For simplicity we restrict the discussion to uniform Cartesian elements; however, the extension to curvilinear elements is straightforward \cite[Appendix B]{Gassner:2016ye}.

First, we subdivide the physical domain $\Omega$ into $N_{\mathrm{el}}$ non-overlapping Cartesian elements $\left\{E_\nu\right\}_{\nu=1}^{N_{\mathrm{el}}}$. Each element is then transformed with a linear mapping, $\spacevec{X}(\spacevec{\xi})$ into reference coordinates $\spacevec{\xi}=(\xi,\eta,\zeta)$ on the element $E_0=[-1,1]^3$ \cite{Gassner:2016ye}. As we restrict to Cartesian meshes, the Jacobian and metric terms are simply
\begin{equation*}
J = \frac{1}{8}\Delta x\Delta y\Delta z,\quad X_{\xi} = \frac{1}{2}\Delta x,\quad Y_{\eta} = \frac{1}{2}\Delta y,\quad Z_{\zeta} = \frac{1}{2}\Delta z,
\end{equation*}
with element side lengths $\Delta x$, $\Delta y$, and $\Delta z$.

For each element, we approximate the components of the state vector, the flux vectors, etc. with polynomials of degree $N$ in each spatial direction. The polynomial approximations are denoted with capital letters, e.g. $\statevec U \in\mathbb{P}^{N}\!(E)$. Here we consider the construction of a nodal discontinuous Galerkin spectral element method (DGSEM), in which the polynomials are written in terms of Lagrange basis functions, e.g. $\ell_i(\xi)$ with $i=0,\ldots,N$, that interpolate at the Lengendre-Gauss-Lobatto (LGL) nodes \cite{Kopriva:2009nx}.

The DG method is built from the weak form of the conservation law \eqref{eq:compactConsLaw} where we multiply by a test function $\varphi\in\mathbb{P}^N$ and integrate of the reference element
\begin{equation}\label{eq:intDG1}
\int\limits_{E_0}\left(J\statevec{U}_t + \spacevec{\nabla}_{\xi}\cdot\bigcontravec{F}\right)\varphi\,\mathrm{d}\spacevec{\xi} = \statevec{0},
\end{equation}
where derivatives are now taken in the reference coordinates and we introduce the contravariant fluxes
\begin{equation*}
\bigcontravec{F} = \begin{bmatrix}
Y_\eta Z_\zeta\statevec{F}_1\\[0.05cm]
X_\xi Z_\zeta\statevec{F}_2\\[0.05cm]
X_\xi Y_\eta\statevec{F}_3\\[0.05cm]
\end{bmatrix}.
\end{equation*}
Note, that there is no continuity of the approximate solution or $\varphi$ assumed between elements boundaries.

We select the test function to be the tensor product basis $\varphi=\ell_i(\xi)\ell_j(\eta)\ell_k(\zeta)$ for $i,j,k=0,\ldots,N$. The integrals in \eqref{eq:intDG1} are approximated with LGL quadrature and we \textit{collocate} the quadrature nodes with the interpolation nodes. This exploits that the Lagrange basis functions are discretely orthogonal and simplifies the nodal DG approximation \cite{Kopriva:2009nx}. The integral approximations introduce the mass matrix
\begin{equation*}
\matrixvec{M} = \mathrm{diag}(\omega_0,\ldots,\omega_N),
\end{equation*}
where $\{\omega_i\}_{i=0}^N$ are the LGL quadrature weights. In addition to the discrete integration matrix, the polynomial basis functions and the interpolation nodes form a discrete polynomial derivative matrix
\begin{equation}\label{eq:derivMat}
\matrixvec{D}_{ij} = \pderivative{\ell_j}{\xi}\bigg|_{\xi=\xi_i},\quad i,j = 0,\ldots,N,
\end{equation}
where $\{\xi_i\}_{i=0}^N$ are the LGL nodes. The derivative matrix \eqref{eq:derivMat} is special as it satisfies the summation-by-parts (SBP) property for all polynomial orders $N$ \cite{gassner_skew_burgers}
\begin{equation*}
\matrixvec{M}\matrixvec{D} + \left(\matrixvec{M}\matrixvec{D}\right)^T = \matrixvec{B} = \mathrm{diag}(-1,0,\ldots,0,1).
\end{equation*}
The SBP property is a discrete equivalent of integration-by-parts that is a crucial component to develop high-order entropy conservative/stable numerical approximations, e.g. \cite{fisher2013,fisher2013_2}

We apply the SBP property once to generate boundary contributions in the approximation of \eqref{eq:intDG1}. Just like in the finite volume method, we resolve the discontinuity across element boundaries with a numerical \textit{surface flux} function, e.g. $\statevec{F}_1^*(-1,\eta_j,\zeta_k;\hat{n})$ where $j,k=0,\ldots,N$ and $\hat{n}$ is the normal vector in reference space. We apply the SBP property again to move discrete derivatives back onto the fluxes inside the volume. Further, if we introduce a two-point numerical \textit{volume flux}, e.g. $\statevec{F}_1^{\#}\left(\statevec{U}_{ijk},\statevec{U}_{mjk}\right)$, that is consistent \eqref{eq:consistencyFV} and symmetric with respect to its arguments, e.g. \cite{Gassner:2016ye}. These steps produce a semi-discrete split form DG approximation
\begin{equation}\label{eq:finalDG}
\resizebox{\textwidth}{!}{$
\begin{aligned}
J\left(\statevec{U}_t\right)_{ijk} &+ \frac{1}{\omega_{N}}\left[\contrastatevec{F}^{*}_1(1,\eta_j,\zeta_k;\hat{n}) - \left(\contrastatevec{F}_1\right)_{Njk}\right] - \frac{1}{\omega_{0}}\left[\contrastatevec{F}_1^{*}(-1,\eta_j,\zeta_k;\hat{n}) - \left(\contrastatevec{F}_1\right)_{0jk}\right]+2\!\sum\limits_{m=0}^N \matrixvec{D}_{im}\contrastatevec{F}_1^{\#}(U_{ijk},U_{mjk})\\
&+\frac{1}{\omega_{N}}\Big[\contrastatevec{F}^{*}_2(\xi_i,1,\zeta_k;\hat{n}) - \left(\contrastatevec{F}_2\right)_{iNk}\Big] - \frac{1}{\omega_{0}}\Big[\contrastatevec{F}_2^{*}(\xi_i,-1,\zeta_k;\hat{n}) - \left(\contrastatevec{F}_2\right)_{i0k}\Big]+2\!\sum\limits_{m=0}^N \matrixvec{D}_{im}\contrastatevec{F}_2^{\#}(U_{ijk},U_{imk})\\
&+ \frac{1}{\omega_{N}}\left[\contrastatevec{F}^{*}_3(\xi_i,\eta_j,1;\hat{n}) - \left(\contrastatevec{F}_3\right)_{ijN}\right] - \frac{1}{\omega_{0}}\left[\contrastatevec{F}_3^{*}(\xi_i,\eta_j,-1;\hat{n}) - \left(\contrastatevec{F}_3\right)_{ij0}\right]+2\!\sum\limits_{m=0}^N \matrixvec{D}_{im}\contrastatevec{F}_3^{\#}(U_{ijk},U_{ijm})=\statevec{0},
\end{aligned}
$}
\end{equation}
for $i,j,k=0,\ldots,N$.

To create an entropy aware high-order DG approximation we select the numerical surface and volume numerical fluxes to be those from the finite volume context \cite{carpenter_esdg,Gassner:2016ye}.

Two variants of the split form DG scheme \eqref{eq:finalDG} are of interest for the polytropic Euler equations:
\begin{enumerate}
\item \textit{Entropy conservative DG} approximation: Select $\statevec{F}^{\#}$ \textit{\textbf{and}} $\statevec{F}^{*}$ to be the entropy conserving fluxes developed in Sect.~\ref{sec:ECFlux}.
\item \textit{Entropy stable DG} approximation: Take $\statevec{F}^{\#}$ to be the entropy conserving fluxes from Sect.~\ref{sec:ECFlux} and $\statevec{F}^{*}$ to be the entropy stable fluxes from Sect.~\ref{sec:ESFlux}.
\end{enumerate}

\section{Numerical results}\label{sec:numRes}

We present numerical tests to validate the theoretical findings of the previous sections for an entropy conservative/stable DG spectral element approximation. To do so, we perform the numerical tests in the two dimensional domain $\Omega = [0,1]^2$. We subdivide the domain into $K$ non-overlapping, uniform Cartesian elements such that the DG approximation takes the form presented in Sect.~\ref{sec:DG}. The semi-discrete scheme \eqref{eq:finalDG} is integrated in time with the explicit five-state, fourth order low storage Runge-Kutta method of Carpenter and Kennedy \cite{Carpenter&Kennedy:1994}. A stable time step is computed according to an adjustable coefficient $CFL\in(0,1]$, the local maximum wave speed, and the relative grid size, e.g. \cite{gassner2011}. For uniform Cartesian meshes the explicit time step is selected by
\begin{equation*}
\Delta t \coloneqq CFL\frac{\Delta x}{\lambda_{\mathrm{max}}(2N+1)}.
\end{equation*} 

First, we will verify the high-order spatial accuracy for the DG scheme with the method of manufactured solutions. For this we assume the solution to polytropic Euler equations takes the form
\begin{equation}\label{eq:manuSol}
\mathbf{u} = \left[h,\,\frac{1}{2}h,\,\frac{3}{2}h\right]^T\quad\mathrm{with}\quad h(x,y,t) = 8 + \cos(2\pi x)\sin(2\pi y)\cos(2\pi t). 
\end{equation}
This introduces an additional residual term on the right hand side of \eqref{eq:compactConsLaw} that reads
\begin{equation}\label{eq:manuResidual}
\mathbf{r} = \begin{bmatrix}
h_t + \frac{1}{2}h_x + \frac{3}{2}h_y\\[0.2cm]
\frac{1}{2}h_t +  \frac{1}{4}h_x + b\rho_x +  \frac{3}{4}h_y\\[0.2cm]
 \frac{1}{2}h_t +  \frac{3}{4}h_x +  \frac{9}{4}h_y + b\rho_y 
\end{bmatrix},
\qquad b =  \begin{cases}
\kappa\gamma \rho^{\gamma-1}, & \mathrm{polytropic}\\[0.1cm]
c^2, & \mathrm{isothermal}
\end{cases}.
\end{equation}
Note that the residual term \eqref{eq:manuResidual} is $\gamma$ dependent.

The second test will demonstrate the entropic properties of the DG approximation. To do so, we use a discontinuous initial condition
\begin{equation}\label{eq:discIni}
\mathbf{u} = \begin{cases}
[1.2,\,0.1,\,0.0]^T, &  x \leq y\\[0.1cm]
[1.0,\, 0.2,\, -0.4]^T, & x>y
\end{cases}.
\end{equation}
To measure the discrete entropy conservation of the DG approximation we examine the entropy residual of the numerical scheme \cite{friedrich2018}. To compute the discrete entropy growth, \eqref{eq:finalDG} is rewritten to be
\begin{equation}\label{eq:DGWithRes}
J\left(\statevec{U}_t\right)_{ij} + \mathbf{Res}\left(\statevec{U}\right)_{ij} = \statevec{0},
\end{equation}
where $i,j = 0,\ldots,N$ and
\begin{equation}\label{eq:DGresidual}
\begin{aligned}
\mathbf{Res}\left(\statevec{U}\right)_{ij} &= \frac{1}{\omega_{N}}\left[\contrastatevec{F}^{*}_1(1,\eta_j;\hat{n}) - \left(\contrastatevec{F}_1\right)_{Nj}\right] - \frac{1}{\omega_{0}}\left[\contrastatevec{F}_1^{*}(-1,\eta_j;\hat{n}) - \left(\contrastatevec{F}_1\right)_{0j}\right]+2\!\sum\limits_{m=0}^N \matrixvec{D}_{im}\contrastatevec{F}_1^{\#}(U_{ij},U_{mj})\\
&+ \frac{1}{\omega_{N}}\Big[\contrastatevec{F}^{*}_2(\xi_i,1;\hat{n}) - \left(\contrastatevec{F}_2\right)_{iN}\Big] - \frac{1}{\omega_{0}}\Big[\contrastatevec{F}_2^{*}(\xi_i,-1;\hat{n}) - \left(\contrastatevec{F}_2\right)_{i0}\Big]+2\!\sum\limits_{m=0}^N \matrixvec{D}_{im}\contrastatevec{F}_2^{\#}(U_{ij},U_{im}).
\end{aligned}
\end{equation}
The growth in discrete entropy is computed by contracting \eqref{eq:DGWithRes} with the entropy variables \eqref{eq:entVars}
\begin{equation*}
J\mathbf{W}_{ij}^T\left(\statevec{U}_t\right)_{ij} = -\statevec{W}_{ij}^T\mathbf{Res}\left(\statevec{U}\right)_{ij} \quad\Leftrightarrow\quad J\left(S_t\right)_{ij} = -\statevec{W}_{ij}^T\mathbf{Res}\left(\statevec{U}\right)_{ij},
\end{equation*}
where we apply the definition of the entropy variables to obtain the temporal derivative, $(S_t)_{ij}$, at each LGL node. The DG approximation is entropy conservative when the two-point finite volume flux from Theorem \ref{thm:EC} is taken to be the interface flux, i.e. $\contrastatevec{F}^{\#}=\contrastatevec{F}^{*}=\contrastatevec{F}^{\ec}$, in \eqref{eq:DGresidual}. This means that
\begin{equation*}
\sum_{\nu=1}^{N_{\mathrm{el}}} J_k\sum_{i=0}^N\sum_{j=0}^N\omega_i\omega_j\left(S_t\right)_{ij} = 0,
\end{equation*}
should hold for all time. We numerically verify this property by computing the integrated residual over the domain $\Omega$
\begin{equation}\label{eq:integratedResid}
IS_t = -\sum_{\nu=1}^{N_{\mathrm{el}}} J_k\sum_{i=0}^N\sum_{j=0}^N\omega_i\omega_j\mathbf{W}_{ij}^T\mathbf{Res}\left(\statevec{U}\right)_{ij},
\end{equation}
and demonstrate that \eqref{eq:integratedResid} is on the order of machine precision for the discontinuous initial condition \eqref{eq:discIni}. If interface dissipation is included, like that described in Sect.~\ref{sec:ESFlux}, the DG approximation is entropy stable and \eqref{eq:integratedResid} becomes
\begin{equation*}
IS_t \leq 0.
\end{equation*}

We consider two particular values of $\gamma$ in these numerical studies: Sect.~\ref{sec:isothermalNumRes} presents results for the isothermal Euler equations where $\gamma=1$ and Sect.~\ref{sec:polytropicNumRes} contains results for the polytropic Euler equations where $\gamma=1.4$. We forgo presenting entropy conservative/stable DG numerical results for the shallow water variant ($\gamma = 2$) as they can be found elsewhere in the literature, e.g. \cite{wintermeyer2017}.

\subsection{Isothermal flow}\label{sec:isothermalNumRes}

Here we take $\gamma=1$ and the speed of sound to be $c=1$.

\subsubsection{Convergence}\label{sec:ConvIsoTest}

We use the manufactured solution \eqref{eq:manuSol} and additional residual \eqref{eq:manuResidual} to investigate the accuracy of the DG approximation for two polynomial orders $N=3$ and $N=4$. Further, we examine the convergence rates of the entropy conservative DGSEM where $\mathbf{F}^{\#}=\mathbf{F}^{*}=\mathbf{F}^{\ec}$ (the flux from Theorem \ref{thm:EC}) as well as the entropy stable DGSEM where $\mathbf{F}^{\#}=\mathbf{F}^{\ec}$ (the flux from Theorem \ref{thm:EC}) and $\mathbf{F}^{*} = \mathbf{F}^{\es}$ (the flux from Theorem \ref{thm:ES}). We run the solution up to a final time of $T=1.0$ We compute the $L^2$ error in the density between the approximation and the manufactured solution of different mesh resolution for each polynomial order. In Table \ref{tab:isoEC} we present the experimental order of convergence (EOC) for the entropy conservative DGSEM. We observe an odd/even effect, that is an EOC of $N$ for odd polynomial orders and $N+1$ for even polynomial orders, which has been previously observed, e.g. \cite{gassner_skew_burgers,hindenlang2019order}. This is particularly noticeable for higher resolution numerical tests. Table \ref{tab:isoES} gives the EOC results for the entropy stable DGSEM where there is no longer an odd/even effect and the convergence rate is $N+1$, as expected for a nodal DG scheme, e.g. \cite{canuto2006,hindenlang2019order}.
	\begin{table}[!ht]
	\centering
	\caption{Experimental order of convergence (EOC) for the \textit{entropy conservative} DG approximation of the isothermal Euler with $c=1$. The $L^2$ error in the density is computed against the manufactured solution \eqref{eq:manuSol} for different mesh resolutions. We run to a final time of $T=1.0$ with $CFL=1$. The convergence results for $N=3$ are given on the left and $N=4$ on the right.}
	\label{tab:isoEC}
	\subfloat[$N=3$]
	{
	\begin{tabular}{@{}lll@{}}
		\toprule\noalign{\smallskip}
		$N_{\mathrm{el}}$ & $L^2$ Error of $\rho$ & EOC  \\
		\noalign{\smallskip}\midrule\noalign{\smallskip}
		4 & 9.8E-02 & ---\\[0.05cm]
		8 & 1.7E-03 & 5.8  \\[0.05cm]
		16 & 1.7E-04 & 3.4 \\[0.05cm]
		32 & 3.4E-05 & 2.3 \\[0.05cm]
		64 & 4.7E-06 & 2.9 \\[0.05cm]
		128 & 6.1E-07 & 3.0\\
		\noalign{\smallskip}\bottomrule
	\end{tabular}
	}
	\qquad\qquad\qquad
	\subfloat[$N=4$]
	{
	\begin{tabular}{@{}lll@{}}
		\toprule\noalign{\smallskip}
		$N_{\mathrm{el}}$ & $L^2$ Error of $\rho$ & EOC  \\
		\noalign{\smallskip}\midrule\noalign{\smallskip}
		4 & 5.0E-03 & ---\\[0.05cm]
		8 & 1.9E-04 & 4.7  \\[0.05cm]
		16 & 2.5E-06 & 6.2 \\[0.05cm]
		32 & 6.0E-08 & 5.4 \\[0.05cm]
		64 & 1.9E-09 & 5.0 \\[0.05cm]
		128 & 8.6E-11 & 4.4\\
		\noalign{\smallskip}\bottomrule
	\end{tabular}
	}
	\end{table}
	\begin{table}[!ht]
	\centering
	\caption{Experimental order of convergence (EOC) for the \textit{entropy stable} DG approximation of the isothermal Euler with $c=1$. The $L^2$ error in the density is computed against the manufactured solution \eqref{eq:manuSol} for different mesh resolutions. We run to a final time of $T=1.0$ with $CFL=1$. The convergence results for $N=3$ are given on the left and $N=4$ on the right.}
	\label{tab:isoES}
	\subfloat[$N=3$]
	{
	\begin{tabular}{@{}lll@{}}
		\toprule\noalign{\smallskip}
		$N_{\mathrm{el}}$ & $L^2$ Error of $\rho$ & EOC  \\
		\noalign{\smallskip}\midrule\noalign{\smallskip}
		4 & 1.3E-02 & ---\\[0.05cm]
		8 & 1.4E-03 & 3.2 \\[0.05cm]
		16 & 1.0E-04 & 3.8 \\[0.05cm]
		32 & 9.5E-06 & 3.4 \\[0.05cm]
		64 & 5.9E-07 & 4.0 \\[0.05cm]
                 128 & 3.6E-08 & 4.0\\
		\noalign{\smallskip}\bottomrule
	\end{tabular}
	}
	\qquad\qquad\qquad
	\subfloat[$N=4$]
	{
	\begin{tabular}{@{}lll@{}}
		\toprule\noalign{\smallskip}
		$N_{\mathrm{el}}$ & $L^2$ Error of $\rho$ & EOC  \\
		\noalign{\smallskip}\midrule\noalign{\smallskip}
		4 & 1.1E-03 & --- \\[0.05cm]
		8 & 6.4E-05 & 4.2  \\[0.05cm]
		16 & 2.2E-06 & 4.9 \\[0.05cm]
		32 & 6.6E-08 & 5.0 \\[0.05cm]
		64 & 2.2E-09 & 4.9 \\[0.05cm]
		128 & 8.6E-11 & 4.7\\
		\noalign{\smallskip}\bottomrule
	\end{tabular}
	}
	\end{table}

\subsubsection{Entropy conservation test}\label{sec:ECisoTest}

For this test we compute the entropy residual \eqref{eq:DGresidual} for two polynomial orders, $N=3$ and $N=4$, and different mesh resolutions. We select the volume and surface fluxes to be $\mathbf{F}^{\#}=\mathbf{F}^{*}=\mathbf{F}^{\ec}$ (the flux from Theorem \ref{thm:EC}). We see in Table \ref{tab:ECiso} that the magnitude of the entropy residual is on the order of machine precision for the discontinuous initial condition \eqref{eq:discIni} for all resolution configurations.
	\begin{table}[!ht]
	\centering
	\caption{Entropy conservation test for the DG approximation of the isothermal Euler with $c=1$. The entropy residual \eqref{eq:integratedResid} is computed for two polynomial orders and different mesh resolutions.}
	\label{tab:ECiso}
	\begin{tabular}{@{}cll@{}}
		\toprule\noalign{\smallskip}
		$N_{\mathrm{el}}$ & $L^2$ Error of $IS_t,\,(N=3)$ & $L^2$ Error of $IS_t,\,(N=4)$\\
		\noalign{\smallskip}\midrule\noalign{\smallskip}
			2 & 8.3E-16  & 4.5E-15 \\[0.05cm]
			4 & 2.1E-15  & 2.1E-14 \\[0.05cm]
			8 & 1.5E-14  & 6.5E-14 \\[0.05cm]
			16 & 7.2E-14 & 2.4E-13 \\[0.05cm]
			32 & 3.2E-13 & 9.1E-13 \\[0.05cm]
			64 & 1.4E-12 & 3.5E-12 \\
		\noalign{\smallskip}\bottomrule
	\end{tabular}
	\end{table}
	
\subsection{Polytropic flow}\label{sec:polytropicNumRes}

Here we take $\gamma=1.4$ and the scaling factor to be $\kappa=0.5$.

\subsubsection{Convergence}

The formulation of the convergence test is very similar to those discussed in Sect.~\ref{sec:ConvIsoTest} where we use the manufactured solution \eqref{eq:manuSol} and additional residual \eqref{eq:manuResidual} to investigate the accuracy of the DG approximation for two polynomial orders $N=3$ and $N=4$. We run the solution up to a final time of $T=1.0$ and compute the $L^2$ error in the density. Table \ref{tab:polyECConv} gives the EOC for the entropy conservative DGSEM where we observe an odd/even effect with respect to the polynomial order of the approximation. The entropy stable EOC results in Table \ref{tab:polyESConv}, again, show optimal convergence order of $N+1$ and no such odd/even effect.
	\begin{table}[!ht]
	\centering
	\caption{Experimental order of convergence (EOC) for the \textit{entropy conservative} DG approximation of the polytropic Euler with $\kappa=0.5$. The $L^2$ error in the density is computed against the manufactured solution \eqref{eq:manuSol} for different mesh resolutions. We run to a final time of $T=1.0$ with $CFL=1$. The convergence results for $N=3$ are given on the left and $N=4$ on the right.}
	\label{tab:polyECConv}
	\subfloat[$N=3$]
	{
	\begin{tabular}{@{}lll@{}}
		\toprule\noalign{\smallskip}
		$N_{\mathrm{el}}$ & $L^2$ Error of $\rho$ & EOC  \\
		\noalign{\smallskip}\midrule\noalign{\smallskip}
			4 & 4.7E-02 & ---  \\[0.05cm]
			8 & 7.1E-03 & 2.7  \\[0.05cm]
			16 & 3.2E-04 & 4.5 \\[0.05cm]
			32 & 1.3E-05 & 4.6 \\[0.05cm]
    			64 & 1.6E-06 & 3.0 \\[0.05cm]
			128 & 2.0E-07 & 2.9\\
		\noalign{\smallskip}\bottomrule
	\end{tabular}
	}
	\qquad\qquad\qquad
	\subfloat[$N=4$]
	{
	\begin{tabular}{@{}lll@{}}
		\toprule\noalign{\smallskip}
		$N_{\mathrm{el}}$ & $L^2$ Error of $\rho$ & EOC  \\
		\noalign{\smallskip}\midrule\noalign{\smallskip}
			4 & 1.5E-02 & ---\\[0.05cm]
			8 & 1.5E-04 & 6.7  \\[0.05cm]
			16 & 4.1E-06 & 5.2 \\[0.05cm]
			32 & 7.2E-08 & 5.8 \\[0.05cm]
			64 & 2.3E-09 & 5.0 \\[0.05cm]
			128 & 8.7E-11 & 4.7\\
		\noalign{\smallskip}\bottomrule
	\end{tabular}
	}
	\end{table}
	\begin{table}[!ht]
	\centering
	\caption{Experimental order of convergence (EOC) for the \textit{entropy stable} DG approximation of the polytropic Euler with $\kappa=0.5$. The $L^2$ error in the density is computed against the manufactured solution \eqref{eq:manuSol} for different mesh resolutions. We run to a final time of $T=1.0$ with $CFL=1$. The convergence results for $N=3$ are given on the left and $N=4$ on the right.}
	\label{tab:polyESConv}
	\subfloat[$N=3$]
	{
	\begin{tabular}{@{}lll@{}}
		\toprule\noalign{\smallskip}
		$N_{\mathrm{el}}$ & $L^2$ Error of $\rho$ & EOC  \\
		\noalign{\smallskip}\midrule\noalign{\smallskip}
			4 & 1.6E-02 & --- \\[0.05cm]
			8 & 1.7E-03 & 3.3  \\[0.05cm]
			16 & 1.5E-04 & 3.5 \\[0.05cm]
			32 & 9.4E-06 & 4.0 \\[0.05cm]
			64 & 6.3E-07 & 3.9 \\[0.05cm]
			128 & 3.9E-08 & 4.0\\
		\noalign{\smallskip}\bottomrule
	\end{tabular}
	}
	\qquad\qquad\qquad
	\subfloat[$N=4$]
	{
	\begin{tabular}{@{}lll@{}}
		\toprule\noalign{\smallskip}
		$N_{\mathrm{el}}$ & $L^2$ Error of $\rho$ & EOC  \\
		\noalign{\smallskip}\midrule\noalign{\smallskip}
			4 & 1.4E-03 & ---  \\[0.05cm]
			8 & 6.2E-05 & 4.5  \\[0.05cm]
			16 & 2.6E-06 & 4.5 \\[0.05cm]
			32 & 7.5E-08 & 5.1 \\[0.05cm]
			64 & 2.5E-09 & 4.9 \\[0.05cm]
			128 & 9.4E-11 & 4.7\\
		\noalign{\smallskip}\bottomrule
	\end{tabular}
	}
	\end{table}

\subsubsection{Entropy conservation test}

Just as in Sect.~\ref{sec:ECisoTest} we compute the entropy residual \eqref{eq:DGresidual} for two polynomial orders, $N=3$ and $N=4$, and different mesh resolutions. The volume and surface fluxes are both taken to be the entropy conservative flux from Theorem \ref{thm:EC}. Table \ref{tab:ECpoly} shows that the entropy residual for the polytropic test is on the order of machine precision for the discontinuous initial condition \eqref{eq:discIni} and all resolution configurations.
	\begin{table}[!ht]
	\centering
	\caption{Entropy conservation test for the DG approximation of the polytropic Euler with $\kappa=0.5$. The entropy residual \eqref{eq:integratedResid} is computed for two polynomial orders and different mesh resolutions.}
	\label{tab:ECpoly}
	\begin{tabular}{@{}cll@{}}
		\toprule\noalign{\smallskip}
		$N_{\mathrm{el}}$ & $L^2$ Error of $IS_t,\,(N=3)$ & $L^2$ Error of $IS_t,\,(N=4)$\\
		\noalign{\smallskip}\midrule\noalign{\smallskip}
			2 & 7.4E-16 & 1.7E-15 \\[0.05cm]
			4 & 1.5E-15 & 9.4E-15 \\[0.05cm]
			8 & 4.7E-15 & 2.8E-14 \\[0.05cm]
			16 & 1.7E-14 & 8.4E-14 \\[0.05cm]
			32 & 6.2E-14 & 3.1E-13 \\[0.05cm]
			64 & 2.4E-13 &1.2E-12 \\
		\noalign{\smallskip}\bottomrule
	\end{tabular}
	\end{table}

\section{Conclusions}

In this work we developed entropy conservative (and entropy stable) numerical approximations for the Euler equations with an equation of state that models a polytropic gas. For this case the pressure is determined from a scaled $\gamma$-power law of the fluid density. In turn, the total energy conservation equation became redundant and it was removed from the polytropic Euler system. In fact, the total energy acted as a mathematical entropy function for the polytropic Euler equations where its conservation (or decay) became an \textit{auxiliary} condition not explicitly modeled by the PDEs. 

We analyzed the continuous entropic properties of the polytropic Euler equations. This provided guidance for the semi-discrete entropy analysis. Next, we derived entropy conservative numerical flux functions in the finite volume context that required the introduction of a special $\gamma$-mean, which is a generalization of the logarithmic mean present in the adiabatic Euler case. Dissipation matrices were then designed and incorporated to guarantee the finite volume fluxes obeyed the entropy inequality discretely. We also investigated two special cases of the polytropic system that can be used to model isothermal gases ($\gamma=1$) or the shallow water equations ($\gamma=2$). The finite volume scheme was extended to high-order spatial accuracy through a specific discontinuous Galerkin spectral element framework. We then validated the theoretical analysis with several numerical results. In particular, we demonstrated the high-order spatial accuracy and entropy conservative/stable properties of the novel numerical fluxes for the polytropic Euler equations.

\section*{Acknowledgements}
Gregor Gassner and Moritz Schily have been supported by the European Research Council (ERC) under the European Union's Eights Framework Program Horizon 2020 with the research project \textit{Extreme}, ERC grant agreement no. 714487.

\bibliographystyle{plain}
\bibliography{references}

\begin{thebibliography}{10}

\bibitem{Barth1999}
Timothy~J. Barth.
\newblock Numerical methods for gasdynamic systems on unstructured meshes.
\newblock In Dietmar Kr\"{o}ner, Mario Ohlberger, and Christian Rohde, editors,
  {\em An Introduction to Recent Developments in Theory and Numerics for
  Conservation Laws}, volume~5 of {\em Lecture Notes in Computational Science
  and Engineering}, pages 195--285. Springer Berlin Heidelberg, 1999.

\bibitem{bohm2018}
Marvin Bohm, Andrew~R. Winters, Gregor~J. Gassner, Dominik Derigs, Florian
  Hindenlang, and Joachim Saur.
\newblock An entropy stable nodal discontinuous {G}alerkin method for the
  resistive {MHD} equations. {Part I}: {T}heory and numerical verification.
\newblock {\em Journal of Computational Physics},
  doi.org/10.1016/j.jcp.2018.06.027, 2018.

\bibitem{canuto2006}
C.~Canuto, M.~Hussaini, A.~Quarteroni, and T.~Zang.
\newblock {\em Spectral Methods: Fundamentals in Single Domains}.
\newblock Springer, Berlin, 2006.

\bibitem{carpenter_esdg}
M.~Carpenter, T.~Fisher, E.~Nielsen, and S.~Frankel.
\newblock Entropy stable spectral collocation schemes for the
  {N}avier--{S}tokes equations: Discontinuous interfaces.
\newblock {\em SIAM Journal on Scientific Computing}, 36(5):B835--B867, 2014.

\bibitem{Carpenter&Kennedy:1994}
M.~Carpenter and C.~Kennedy.
\newblock Fourth-order $2{N}$-storage {R}unge-{K}utta schemes.
\newblock Technical Report NASA TM 109111, NASA Langley Research Center, 1994.

\bibitem{cengel2014}
Yunus Cengel and Michael Boles.
\newblock {\em Thermodynamics: An Engineering Approach}.
\newblock McGraw-Hill Education; 8 edition, 2014.

\bibitem{chan2018}
Jesse Chan.
\newblock On discretely entropy conservative and entropy stable discontinuous
  {G}alerkin methods.
\newblock {\em {Journal of Computational Physics}}, 362:346--374, 2018.

\bibitem{Chandrashekar2012}
Praveen Chandrashekar.
\newblock {Kinetic Energy Preserving and Entropy Stable Finite Volume Schemes
  for Compressible Euler and Navier-Stokes Equations}.
\newblock {\em Communications in Computational Physics}, 14:1252--1286, 2013.

\bibitem{chen2005}
Gui-Qiang Chen.
\newblock Euler equations and related hyperbolic conservation laws.
\newblock In {\em Handbook of differential equations: evolutionary equations},
  volume~2, pages 1--104, 2005.

\bibitem{Chen2017}
Tianheng Chen and Chi-Wang Shu.
\newblock Entropy stable high order discontinuous {G}alerkin methods with
  suitable quadrature rules for hyperbolic conservation laws.
\newblock {\em {Journal of Computational Physics}}, 345:427--461, 2017.

\bibitem{courant1967}
Richard Courant, Kurt Friedrichs, and Hans Lewy.
\newblock On partial differential equations of mathematical physics.
\newblock {\em IBM Journal of Reseach and Development}, 11:215--234, 1967.

\bibitem{crean2018}
Jared Crean, Jason~E. Hicken, David C. Del~Rey Fern{\'a}ndez, David~W. Zingg,
  and Mark~H. Carpenter.
\newblock Entropy-stable summation-by-parts discretization of the {E}uler
  equations on general curved elements.
\newblock {\em {Journal of Computational Physics}}, 356:410--438, 2018.

\bibitem{Derigs2016_2}
Dominik Derigs, Andrew~R. Winters, Gregor~J. Gassner, and Stefanie Walch.
\newblock A novel averaging technique for discrete entropy stable dissipation
  operators for ideal {MHD}.
\newblock {\em Journal of Computational Physics}, 330:624--632, 2016.

\bibitem{evans2010}
Laurence~C. Evans.
\newblock {\em Partial Differential Equations}.
\newblock American Mathematical Society, 2012.

\bibitem{fisher2013}
Travis~C. Fisher and Mark~H. Carpenter.
\newblock High-order entropy stable finite difference schemes for nonlinear
  conservation laws: {F}inite domains.
\newblock {\em Journal of Computational Physics}, 252:518--557, 2013.

\bibitem{fisher2013_2}
Travis~C. Fisher, Mark~H. Carpenter, Jan Nordstr\"{o}m, Nail~K. Yamaleev, and
  Charles Swanson.
\newblock Discretely conservative finite-difference formulations for nonlinear
  conservation laws in split form: {T}heory and boundary conditions.
\newblock {\em Journal of Computational Physics}, 234:353--375, 2013.

\bibitem{Fjordholm2011}
Ulrik~S. Fjordholm, Siddhartha Mishra, and Eitan Tadmor.
\newblock Well-blanaced and energy stable schemes for the shallow water
  equations with discontiuous topography.
\newblock {\em Journal of Computational Physics}, 230(14):5587--5609, 2011.

\bibitem{Fjordholm2012_2}
Ulrik~S. Fjordholm, Siddhartha Mishra, and Eitan Tadmor.
\newblock Arbitrarily high-order accurate entropy stable essentially
  nonoscillatory schemes for systems of conservation laws.
\newblock {\em {SIAM} Journal on Numerical Analysis}, 50(2):544--573, 2012.

\bibitem{Fjordholm2016}
Ulrik~S. Fjordholm and Deep Ray.
\newblock A sign preserving {WENO} reconstruction method.
\newblock {\em {Journal of Scientific Computing}}, 68(1):42--63, 2016.

\bibitem{friedrich2018}
Lucas Friedrich, Andrew~R Winters, David C Del~Rey Fern{\'a}ndez, Gregor~J
  Gassner, Matteo Parsani, and Mark~H Carpenter.
\newblock An entropy stable $h/p$ non-conforming discontinuous {G}alerkin
  method with the summation-by-parts property.
\newblock {\em Journal of Scientific Computing}, 77(2):689--725, 2018.

\bibitem{gassner2011}
G.~Gassner, F.~Hindenlang, and C.~Munz.
\newblock A {R}unge-{K}utta based discontinuous {G}alerkin method with time
  accurate local time stepping.
\newblock {\em Adaptive High-Order Methods in Computational Fluid Dynamics},
  2:95--118, 2011.

\bibitem{gassner_skew_burgers}
Gregor~J. Gassner.
\newblock A skew-symmetric discontinuous {Galerkin} spectral element
  discretization and its relation to {SBP-SAT} finite difference methods.
\newblock {\em SIAM Journal on Scientific Computing}, 35(3):A1233--A1253, 2013.

\bibitem{Gassner2017}
Gregor~J Gassner, Andrew~R Winters, Florian~J. Hindenlang, and David~A.
  Kopriva.
\newblock The {BR1} scheme is stable for the compressible {N}avier-{S}tokes
  equations.
\newblock {\em {Journal of Scientific Computing}}, 77(1):154--200, 2017.

\bibitem{Gassner:2016ye}
Gregor~J. Gassner, Andrew~R. Winters, and David~A. Kopriva.
\newblock Split form nodal discontinuous {G}alerkin schemes with
  summation-by-parts property for the compressible {E}uler equations.
\newblock {\em {Journal Of Computational Physics}}, 327:39--66, 2016.

\bibitem{harten1983}
Amiram Harten.
\newblock On the symmetric form of systems of conservation laws with entropy.
\newblock {\em {Journal of Computational Physics}}, 49:151--164, 1983.

\bibitem{hindenlang2019order}
Florian~J Hindenlang and Gregor~J Gassner.
\newblock On the order reduction of entropy stable {DGSEM} for the compressible
  {E}uler equations.
\newblock {\em arXiv preprint arXiv:1901.05812}, 2019.

\bibitem{IsmailRoe2009}
Farzad Ismail and Philip~L. Roe.
\newblock Affordable, entropy-consistent {E}uler flux functions {II}: Entropy
  production at shocks.
\newblock {\em Journal of Computational Physics}, 228(15):5410--5436, 2009.

\bibitem{Kopriva:2009nx}
David~A. Kopriva.
\newblock {\em Implementing Spectral Methods for Partial Differential
  Equations}.
\newblock Scientific Computation. Springer, May 2009.

\bibitem{kundu2008}
Pijush~K. Kundu, Ira~M. Cohen, and D.~W. Dowling.
\newblock {\em Fluid Mechanics 4th}.
\newblock Elsevier, Oxford, 2008.

\bibitem{kuya2018}
Yuichi Kuya, Kosuke Totani, and Soshi Kawai.
\newblock Kinetic energy and entropy preserving schemes for compressible flows
  by split convective forms.
\newblock {\em {Journal of Computational Physics}}, 375:823--853, 2018.

\bibitem{Lefloch2002}
P.~G. LeFloch, J.~M. Mercier, and C.~Rohde.
\newblock Fully discrete, entropy conservative schemes of arbitraryorder.
\newblock {\em SIAM Journal on Numerical Analysis}, 40(5):1968--1992, 2002.

\bibitem{leveque2002}
Randall~J. Le{V}eque.
\newblock {\em Finite Volume Methods for Hyperbolic Problems}, volume~31.
\newblock Cambridge University Press, 2002.

\bibitem{maxima}
Maxima.
\newblock Maxima, a computer algebra system. version 5.42.2, 2019.

\bibitem{mock1980}
Michael~S. Mock.
\newblock Systems of conservation laws of mixed type.
\newblock {\em Journal of Differential Equations}, 37(1):70--88, 1980.

\bibitem{Ray2016}
Deep Ray, Praveen Chandrashekar, Ulrik~S. Fjordholm, and Siddhartha Mishra.
\newblock Entropy stable scheme on two-dimensional unstructured grids for
  {E}uler equations.
\newblock {\em Communications in Computational Physics}, 19(5):1111--1140,
  2016.

\bibitem{serre1999}
Denis Serre.
\newblock {\em Systems of Conservation Laws 1: Hyperbolicity, entropies, shock
  waves}.
\newblock Cambridge University Press, 1999.

\bibitem{stolarsky1975}
Kenneth~B. Stolarsky.
\newblock Generalizations of the logarithmic mean.
\newblock {\em Mathematics Magazine}, 48(2):87--92, 1975.

\bibitem{stolarsky1980}
Kenneth~B. Stolarsky.
\newblock The power and generalized logarithmic means.
\newblock {\em American Mathematical Monthly}, 87(7):545--548, 1980.

\bibitem{svard2014}
Magnus Sv\"{a}rd and Jan Nordstr\"{o}m.
\newblock Review of summation-by-parts schemes for initial-boundary-value
  problems.
\newblock {\em {Journal of Computational Physics}}, 268:17--38, 2014.

\bibitem{tadmor1984}
Eitan Tadmor.
\newblock Skew-selfadjoint form for systems of conservation laws.
\newblock {\em Journal of Mathematical Analysis and Applications},
  103(2):428--442, 1984.

\bibitem{Tadmor1987}
Eitan Tadmor.
\newblock The numerical viscosity of entropy stable schemes for systems of
  conservation laws.
\newblock {\em Mathematics of Computation}, 49(179):91--103, 1987.

\bibitem{Tadmor2003}
Eitan Tadmor.
\newblock Entropy stability theory for difference approximations of nonlinear
  conservation laws and related time-dependent problems.
\newblock {\em Acta Numerica}, 12:451--512, 5 2003.

\bibitem{toro2009}
Eleuterio~F. Toro.
\newblock {\em Riemann Solvers and Numerical Methods for Fluid Dynamics: A
  Practical Introduction}.
\newblock Springer, 2009.

\bibitem{whitham1974}
G.~B. Whitham.
\newblock {\em Linear and Nonlinear Waves}.
\newblock John Wiley and Sons, New York, 1974.

\bibitem{wintermeyer2017}
Niklas Wintermeyer, Andrew~R. Winters, Gregor~J. Gassner, and David~A. Kopriva.
\newblock An entropy stable discontinuous {G}alerkin method for the two
  dimensional shallow water equations with discontinuous bathymetry.
\newblock {\em {Journal of Computational Physics}}, 340:200--242, 2017.

\bibitem{Winters2017}
Andrew~R. Winters, Dominik Derigs, Gregor~J. Gassner, and Stefanie Walch.
\newblock A uniquely defined entropy stable matrix dissipation operator for
  high {M}ach number ideal {MHD} and compressible {E}uler simulations.
\newblock {\em {Journal of Computational Physics}}, 332:274--289, 2017.

\bibitem{Winters2016}
Andrew~R. Winters and Gregor~J. Gassner.
\newblock {Affordable, Entropy Conserving and Entropy Stable Flux Functions for
  the Ideal {MHD} Equations}.
\newblock {\em Journal of Computational Physics}, 304:72--108, 2016.

\bibitem{Winters2018}
Andrew~R Winters, Rodrigo~C Moura, Gianmarco Mengaldo, Gregor~J Gassner,
  Stefanie Walch, Joaquim Peiro, and Spencer~J Sherwin.
\newblock A comparative study on polynomial dealiasing and split form
  discontinuous {G}alerkin schemes for under-resolved turbulence computations.
\newblock {\em {Journal of Computational Physics}}, 372:1--21, 2018.

\end{thebibliography}
\appendix

\section{Numerically stable procedures to evaluate special averages}

In the derivation of the entropy conservative flux function and the discrete evaluations of the dissipation matrices needed to create an entropy stable numerical flux we encountered two special average states, one for the fluid density \eqref{eq:gammaAvg} and one for the square of the sound speed \eqref{eq:soundSpeedAvg}. From their construction, both averages can exhibit numerical instabilities when the left and right state values approach one another because the averages tend to a $0/0$ form. Therefore, we provide numerically stable evaluations of these averages when the left and right states are deemed ``close.''

\subsection{$\gamavg$ procedure}\label{app:gamAvgEval}

We start from the definition of the $\gamma$-mean necessary for the polytropic Euler entropy conservative fluxes. It is defined as
\begin{equation}\label{eq:gammaAvgApp}
\gamavg = \frac{1}{\gamma}\frac{\jump{p}}{\jump{e}},
\end{equation}
where $p$ is the pressure and $e$ is the internal energy. For the isothermal case ($\gamma=1$) and the shallow water case ($\gamma=2$) the $\gamma$-mean is well defined as discussed in Def. \ref{def:gamAvg}. So, we only investigate the $\gamma>1$, $\gamma\ne 2$ case. From \eqref{eq:eos} and \eqref{eq:innerEnergy} we know that the $\gamma$-mean takes the form
\begin{equation*}
\gamavg = \frac{\gamma-1}{\gamma}\frac{\jump{\rho^\gamma}}{\jump{\rho^{\gamma-1}}}.
\end{equation*}
To develop a numerically stable approach when $\rho_R\approx\rho_L$ we first rewrite the left and right states to be
\begin{equation*}
\jump{\rho^\gamma} = \rho_R^\gamma - \rho_L^\gamma
= \left(\avg{\rho} + \frac{1}{2}\jump{\rho}\right)^{\!\!\gamma} - \left(\avg{\rho} - \frac{1}{2}\jump{\rho}\right)^{\!\!\gamma}
= \avg{\rho}^{\!\gamma}\left[\left(1 + \frac{\jump{\rho}}{2\avg{\rho}}\right)^{\!\!\gamma} - \left(1 - \frac{\jump{\rho}}{2\avg{\rho}}\right)^{\!\!\gamma}\right].
\end{equation*}
To simplify the discussion we introduce an auxiliary variable
\begin{equation}\label{eq:auxVarF}
f = \frac{\jump{\rho}}{2\avg{\rho}} = \frac{\rho_R-\rho_L}{\rho_R+\rho_L},
\end{equation}
such that
\begin{equation*}
\jump{\rho^\gamma} = \avg{\rho}^{\!\gamma}\left[(1+f)^{\gamma} - (1-f)^{\gamma}\right],
\end{equation*}
and
\begin{equation*}
\jump{\rho^{\gamma-1}} = \avg{\rho}^{\!\gamma-1}\left[(1+f)^{\gamma-1} - (1-f)^{\gamma-1}\right].
\end{equation*}
Then, we rewrite the $\gamma$-mean in the form
\begin{equation*}
\avg{\rho}_{\!\gamma} = \frac{\gamma-1}{\gamma}\frac{\avg{\rho}^{\!\gamma}}{\avg{\rho}^{\!\gamma-1}}\frac{(1+f)^{\gamma} - (1-f)^{\gamma}}{(1+f)^{\gamma-1} - (1-f)^{\gamma-1}}
= \avg{\rho}\left[\frac{\left(\frac{(1+f)^{\gamma} - (1-f)^{\gamma}}{\gamma}\right)}{\left(\frac{(1+f)^{\gamma-1} - (1-f)^{\gamma-1}}{\gamma-1}\right)}\right]
\end{equation*}
Next, we perform several Taylor expansions (essentially applying the Binomial Theorem) to find
\begin{equation}\label{eq:Taylor1}
\resizebox{0.9\textwidth}{!}{$
\begin{aligned}
(1+f)^\gamma &= 1+\gamma f + \frac{\gamma(\gamma-1)}{2!}f^2+\frac{\gamma(\gamma-1)(\gamma-2)}{3!}f^3+\frac{\gamma(\gamma-1)(\gamma-2)(\gamma-3)}{4!}f^4 + \mathcal{O}(f^5),\\[0.2cm]
(1-f)^\gamma &= 1-\gamma f + \frac{\gamma(\gamma-1)}{2!}f^2-\frac{\gamma(\gamma-1)(\gamma-2)}{3!}f^3+\frac{\gamma(\gamma-1)(\gamma-2)(\gamma-3)}{4!}f^4 + \mathcal{O}(f^5),\\[0.2cm]
(1+f)^{\!\gamma-1} &= 1+(\gamma-1) f + \frac{(\gamma-1)(\gamma-2)}{2!}f^2+\frac{(\gamma-1)(\gamma-2)(\gamma-3)}{3!}f^3+\frac{(\gamma-1)(\gamma-2)(\gamma-3)(\gamma-4)}{4!}f^4 + \mathcal{O}(f^5),\\[0.2cm]
(1+f)^{\!\gamma-1} &= 1-(\gamma-1) f + \frac{(\gamma-1)(\gamma-2)}{2!}f^2-\frac{(\gamma-1)(\gamma-2)(\gamma-3)}{3!}f^3+\frac{(\gamma-1)(\gamma-2)(\gamma-3)(\gamma-4)}{4!}f^4 + \mathcal{O}(f^5).
\end{aligned}$}
\end{equation}
Then we compute
\begin{equation*}
\frac{(1+f)^{\gamma} - (1-f)^{\gamma}}{\gamma} = 2f + \frac{2(\gamma-1)(\gamma-2)}{3!}f^3 + \frac{2(\gamma-1)(\gamma-2)(\gamma-3)(\gamma-4)}{5!}f^5 + \mathcal{O}(f^7),
\end{equation*}
and
\begin{equation*}
\frac{(1+f)^{\gamma-1} - (1-f)^{\gamma-1}}{\gamma-1} = 2f + \frac{2(\gamma-2)(\gamma-3)}{3!}f^3 + \frac{2(\gamma-2)(\gamma-3)(\gamma-4)(\gamma-5)}{5!}f^5 + \mathcal{O}(f^7).
\end{equation*}
Now, we use polynomial long division to find
\begin{equation*}
\frac{\left(\frac{(1+f)^{\gamma} - (1-f)^{\gamma}}{\gamma}\right)}{\left(\frac{(1+f)^{\gamma-1} - (1-f)^{\gamma-1}}{\gamma-1}\right)} = 1 + \frac{\gamma-2}{3}f^2 - \frac{(\gamma+1)(\gamma-2)(\gamma-3)}{45}f^4 + \frac{(\gamma+1)(\gamma-2)(\gamma-3)(2\gamma(\gamma-2)-9)}{945}f^6 + \mathcal{O}(f^8)
\end{equation*}
So, we have successfully rewritten the $\gamma$-mean to take a more numerically stable form
\begin{equation}\label{eq:gammaMeanFinal}
\resizebox{0.9\textwidth}{!}{$
\gamavg = \avg{\rho}\left(1 + \frac{\gamma-2}{3}f^2 - \frac{(\gamma+1)(\gamma-2)(\gamma-3)}{45}f^4 + \frac{(\gamma+1)(\gamma-2)(\gamma-3)(2\gamma(\gamma-2)-9)}{945}f^6 + \mathcal{O}(f^8)\right)
$}
\end{equation}

Interestingly, we see that the $\gamma=2$ case remains correct as well because $(\gamma-2)$ is a common factor of the higher order terms and the expression becomes the arithmetic mean. For the isothermal case we set $\gamma=1$ in \eqref{eq:gammaMeanFinal} and obtain
\begin{equation}\label{eq:proposedLogMean}
\gamavg\!\bigg|_{\gamma=1} =  \avg{\rho}\left(1 - \frac{1}{3}f^2 - \frac{4}{45}f^4 - \frac{44}{945}f^6 + \mathcal{O}(f^8)\right).
\end{equation}
The expansion \eqref{eq:proposedLogMean} is a numerically stable evaluation of the \textit{logarithmic mean} that is equivalent to the expansion derived by Ismail and Roe \cite{IsmailRoe2009}. We see that there is no longer any issue for the isothermal limit $\gamma\rightarrow 1$ because the expansion and subsequent division removed any $0/0$ tendencies.
In fact, the expansion \eqref{eq:gammaMeanFinal} is valid for any value of $\gamma$ provided the left and right states are close in value.

We outline the algorithm to compute the $\gamma$-mean:
\begin{equation*}
\setlength{\belowdisplayskip}{0pt}
f = \frac{\rho_R-\rho_L}{\rho_R+\rho_L},\quad \nu = f^2
\end{equation*}
\begin{itemize}
\item[] \textbf{if} $(\nu<10^{-4})$ \textbf{then}
    	\begin{equation*}
    	\setlength{\abovedisplayskip}{0pt}
    	\setlength{\belowdisplayskip}{0pt}
	\gamavg = \avg{\rho}\left(1 + \nu\left(\frac{\gamma-2}{3} - \nu\left(\frac{(\gamma+1)(\gamma-2)(\gamma-3)}{45} + \nu\frac{(\gamma+1)(\gamma-2)(\gamma-3)(2\gamma(\gamma-2)-9)}{945}\right)\right)\right)
	\end{equation*}
\item[] \textbf{else}\newline {\color{white}{aaal}}Use formula \eqref{eq:gammaAvgApp}
\end{itemize}

\subsection{$\aAvg$ procedure}\label{app:aAvgEval}

For the polytropic Euler equations we are able to find a unique averaging procedure for the dissipation matrices, but require a specific average of the sound speed
\begin{equation}\label{eq:soundSpeedApp}
\overline{a^2} = \frac{\jump{p}}{\jump{\rho}} = \kappa\frac{\jump{\rho^\gamma}}{\jump{\rho}}.
\end{equation}
Again, we consider $\gamma>1$, $\gamma\ne 2$ as those cases are numerically stable as discussed in Rem. \ref{rem:soundSpeedAvg}. When the left and right density states are close we require a numerically stable way to evaluate this average. Just as in Appendix \ref{app:gamAvgEval} we rewrite the jumps and perform Taylor expansions
\begin{equation*}
\begin{aligned}
\jump{\rho^\gamma} &= \avg{\rho}^\gamma\left[(1+f)^\gamma-(1-f)^\gamma\right],\\[0.05cm]
\jump{\rho} &= \rho_R-\rho_L = \avg{\rho} + \frac{1}{2}\jump{\rho} - \left(\avg{\rho}-\frac{1}{2}\jump{\rho}\right) = 2\avg{\rho}f,
\end{aligned}
\end{equation*}
where $f$ is defined as in \eqref{eq:auxVarF}. Therefore, the average of the sound speed becomes
\begin{equation*}
\overline{a^2} = \kappa\frac{\jump{\rho^\gamma}}{\jump{\rho}} = \kappa\frac{ \avg{\rho}^\gamma\left[(1+f)^\gamma-(1-f)^\gamma\right]}{2\avg{\rho}f}=\kappa\avg{\rho}^{\gamma-1}\frac{(1+f)^\gamma-(1-f)^\gamma}{2f}.
\end{equation*}
We apply the first two Taylor expansions from \eqref{eq:Taylor1} and find
\begin{equation*}
\begin{aligned}
\overline{a^2} = \gamma\kappa\avg{\rho}^{\gamma-1}&\left(1 + \frac{(\gamma-1)(\gamma-2)}{6}f^2 + \frac{(\gamma-1)(\gamma-2)(\gamma-3)(\gamma-4)}{120}f^4\right.\\
&\qquad\left.+\frac{(\gamma-1)(\gamma-2)(\gamma-3)(\gamma-4)(\gamma-5)(\gamma-6)}{5040}f^6+\mathcal{O}(f^8)\right).
\end{aligned}
\end{equation*}
The square sound speed average is computed with the same logic as the $\gamma$-mean from Appendix \ref{app:gamAvgEval}:
\begin{equation*}
\setlength{\belowdisplayskip}{0pt}
f = \frac{\rho_R-\rho_L}{\rho_R+\rho_L},\quad \nu = f^2
\end{equation*}
\begin{itemize}
\item[] \textbf{if} $(\nu<10^{-4})$ \textbf{then}
\begin{equation*}
\setlength{\abovedisplayskip}{1.5pt}
\setlength{\belowdisplayskip}{0pt}
\resizebox{0.85\textwidth}{!}{$
\overline{a^2} = \gamma\kappa\avg{\rho}^{\gamma-1}\left(1 + \nu\left(\frac{(\gamma-1)(\gamma-2)}{6} + \nu\left(\frac{(\gamma-1)(\gamma-2)(\gamma-3)(\gamma-4)}{120}+\nu\frac{(\gamma-1)(\gamma-2)(\gamma-3)(\gamma-4)(\gamma-5)(\gamma-6)}{5040}\right)\right)\right)$}
\end{equation*}
\item[] \textbf{else}\newline {\color{white}{aaaaal}}Use formula \eqref{eq:soundSpeedApp}
\end{itemize}

\section{Multi-dimensional EC/ES fluxes}\label{app:3DFlux}

Here we collect the entropy conservative numerical flux functions in the $y$ and $z$ directions. The proof of entropy conservation is very similar to the proof of Thm. \ref{thm:EC}, although the discrete entropy conservation condition changes for the two other spatial directions to be 
\begin{equation*}
\jump{\statevec{w}}^T\statevec{f}_2^* = \jump{\Psi_2},\quad\mathrm{and}\quad\jump{\statevec{w}}^T\statevec{f}_3^* = \jump{\Psi_3},
\end{equation*}
respectively. The entropy flux potentials are
\begin{equation*}
\Psi_2 = pv_2,\qquad \Psi_2 = pv_3.
\end{equation*}
\begin{corollary}[Entropy conservative $y$ and $z$ fluxes]\label{cor:ECyz}
The entropy conservative numerical flux functions in the $y$ and $z$ directions are
\begin{equation*}
\statevec{f}_2^{*,\ec} = \begin{pmatrix}
\gamavg\avg{v_2}\\[0.1cm]
\gamavg\avg{v_1}\avg{v_2} \\[0.1cm]
\gamavg\avg{v_2}^2 + \avg{p} \\[0.1cm]
\gamavg\avg{v_2}\avg{v_3}\\[0.1cm]
\end{pmatrix},
\qquad
\statevec{f}_3^{*,\ec} = \begin{pmatrix}
\gamavg\avg{v_3}\\[0.1cm]
\gamavg\avg{v_1}\avg{v_3}\\[0.1cm]
\gamavg\avg{v_2}\avg{v_3}\\[0.1cm]
\gamavg\avg{v_3}^2 + \avg{p}\\[0.1cm]
\end{pmatrix}.
\end{equation*}
\end{corollary}

Next, we give the entropy stable numerical flux functions in the other spatial coordinate directions. To do so, we give the appropriate discrete versions of the right eigenvector, eigenvalue, and diagonal scaling matrices which follow directly from Lemma \ref{lem:RZ}. The entropy stable numerical fluxes are constructed similarly to those given in Thm. \ref{thm:ES}.
\begin{corollary}[Entropy stable $y$ and $z$ fluxes]
The entropy stable flux functions in the $y$ and $z$ directions are
\begin{equation*}
\statevec{f}^{*,\es}_{2} = \statevec{f}^{*,\ec}_2 - \frac{1}{2}\hat{\nineMatrix{R}}_2|\hat{\nineMatrix{\Lambda}}_2|\hat{\nineMatrix{Z}}\hat{\nineMatrix{R}}_2^T\jump{\statevec{w}},\qquad\statevec{f}^{*,\es}_{3} = \statevec{f}^{*,\ec}_3 - \frac{1}{2}\hat{\nineMatrix{R}}_3|\hat{\nineMatrix{\Lambda}}_3|\hat{\nineMatrix{Z}}\hat{\nineMatrix{R}}_3^T\jump{\statevec{w}},
\end{equation*}
where $\statevec{f}^{*,\ec}_{2,3}$ come from Cor. \ref{cor:ECyz}, $\jump{\statevec{w}}$ is given in \eqref{eq:jumpEntVars}, and the appropriate dissipation matrices
\begin{equation*}
\begin{aligned}
\hat{\nineMatrix{R}}_2 &= \begin{bmatrix}
1 & 0 & 0 & 1\\[0.1cm]
\avg{v_1}  & 1 & 0 & \avg{v_1} \\[0.1cm]
\avg{v_2} - \sqrt{\aAvg}& 0 & 0 & \avg{v_2}+ \sqrt{\aAvg}\\[0.1cm]
\avg{v_3} & 0 & 1 & \avg{v_3}\\[0.1cm]
\end{bmatrix},
\quad
\hat{\nineMatrix{R}}_3 = \begin{bmatrix}
1 & 0 & 0 & 1\\[0.1cm]
\avg{v_1}  & 1 & 0 & \avg{v_1} \\[0.1cm]
\avg{v_2} & 0 & 1 & \avg{v_2}\\[0.1cm]
\avg{v_3} - \sqrt{\aAvg}& 0 & 0 & \avg{v_3}+ \sqrt{\aAvg}\\[0.1cm]
\end{bmatrix},\\[0.05cm]
\hat{\nineMatrix{\Lambda}}_2 &= \mathrm{diag}\left(\avg{v_2}-\sqrt{\aAvg}\,,\,\avg{v_2}\,,\,\avg{v_2}\,,\,\avg{v_2}+\sqrt{\aAvg}\right),\\[0.05cm]
\hat{\nineMatrix{\Lambda}}_3 &= \mathrm{diag}\left(\avg{v_3}-\sqrt{\aAvg}\,,\,\avg{v_3}\,,\,\avg{v_3}\,,\,\avg{v_3}+\sqrt{\aAvg}\right),\\[0.05cm]
\hat{\nineMatrix{Z}} &= \mathrm{diag}\left(\frac{\gamavg}{2\aAvg}\,,\,\gamavg\,,\,\gamavg\,,\,\frac{\gamavg}{2\aAvg}\right).
\end{aligned}
\end{equation*}
\end{corollary}
\end{document}